\documentclass[12pt, reqno]{amsart}
\usepackage{floatrow}
\usepackage{amssymb, amsfonts, amsthm, amsmath, amscd}
\usepackage{graphicx}
\usepackage[latin2]{inputenc}
\usepackage{t1enc}
\usepackage[mathscr]{eucal}
\usepackage{indentfirst}
\usepackage{pict2e}
\usepackage{epic}
\numberwithin{equation}{section}
\usepackage{stmaryrd}
\usepackage{scalerel,stackengine}
\usepackage[margin=2.5cm]{geometry}
\usepackage{mathtools}
\usepackage[colorlinks = true,
            linkcolor = blue,
            urlcolor  = blue,
            citecolor = blue,
            anchorcolor = blue]{hyperref}
\usepackage[titletoc]{appendix}
\usepackage[T1]{fontenc}
\usepackage[numbers,square]{natbib}

\usepackage{adjustbox}
\usepackage{dsfont}
\usepackage{times}
\usepackage{enumerate}
\usepackage{setspace}
\usepackage{leftidx}
\usepackage{multirow}

\usepackage{tikz-cd}
\usepackage{pgf,tikz}
\usetikzlibrary{math}
\usetikzlibrary{arrows,decorations.markings}
\usetikzlibrary{patterns,intersections}
\usetikzlibrary{calc}
\usetikzlibrary{decorations.pathreplacing}

\newenvironment{myproof}[2]{\paragraph{\textit{Proof of {#1} }{#2}.}}{\hfill$\square$}

\theoremstyle{plain}
\newtheorem{theorem}{Theorem}[section]
\newtheorem{lemma}[theorem]{Lemma}
\newtheorem{corollary}[theorem]{Corollary}
\newtheorem{proposition}[theorem]{Proposition}

\theoremstyle{definition}
\newtheorem{definition}[theorem]{Definition}
\newtheorem{assume}[theorem]{Assumption}
\newtheorem{remark}[theorem]{Remark}

\newtheorem{condition}[theorem]{Condition}
\newtheorem{problem}[theorem]{Problem}
\newtheorem{example}[theorem]{Example}

\def\CC{\mathbb{C}}

\def\RR{\mathbb{R}}
\def\ZZ{\mathbb{Z}}

\def\a{\alpha}
\def\b{\beta}
\def\g{\gamma}

\def\G{\Gamma}
\def\w{\omega}
\def\l{\lambda}
\def\ra{\rightarrow}
\def\tt{\theta}
\def\zt{\zeta}

\def\mb{\mathbf}
\def\ol{\overline}
\def\O{\Omega}
\def\L{L}

\def\gg{\mathfrak{g}}
\def\kk{\mathfrak{k}}
\def\pp{\mathfrak{p}}
\def\UU{\mathcal{U}}
\def\t{\mathfrak{t}}

\def\C{C_M}

\def\pj{\pi_{SQ}\circ\phi}
\def\X{\mathcal{X}}
\def\Ham{Hamiltonian }
\def\e{\epsilon}
\def\J{\mathcal{J}}
\def\ul{\underline}
\def\S{\Sigma}
\def\sle{strip-like end }
\def\pt{\partial}
\def\I{\mathcal{I}}

\def\DH{D_{\mathbb{H}}}

\def\mr{\mathring}
\def\AA{\mathcal{A}}
\def\ma{m_{\a}}
\newcommand{\maaa}[1]{m_{#1}}
\def\CO{\mathcal{O}}
\newcommand{\Ra}[1]{R^+_{#1}}

\newcommand{\ka}[1]{\mathfrak{k}_{#1,\a}}
\newcommand{\pa}[1]{\mathfrak{p}_{#1,\a}}
\newcommand{\cent}[2]{C(#1,#2)}
\newcommand{\norm}[2]{N(#1,#2)}
\def\wc{\mathcal{C}}
\def\RF{\Lambda}
\newcommand{\RFF}[1]{\Lambda_{#1}}
\newcommand{\RFFF}[1]{\Lambda_{#1}}
\def\monoconst{\tau}
\def\sq{\pi_{SQ}}
\def\coeff{\Bbbk}
\def\capping{moment capping }
\def\afflag{\mathbb{L}}
\def\involute{\sigma}
\def\TTT{\mathbb{T}}
\def\rmor{\varphi}
\def\VT{T_{\sq}^v\C}
\def\VTT{T_{\pi}^v\C}

\DeclareMathOperator{\im}{Im}
\DeclareMathOperator{\ind}{Ind}
\DeclareMathOperator{\diam}{diam}
\DeclareMathOperator{\Lag}{Lag}
\DeclareMathOperator{\id}{id}
\DeclareMathOperator{\ad}{ad}
\DeclareMathOperator{\Aut}{Aut}
\DeclareMathOperator{\pr}{pr}
\DeclareMathOperator{\Hor}{Hor}
\DeclareMathOperator{\Ver}{Ver}
\DeclareMathOperator{\Ad}{Ad}
\DeclareMathOperator{\area}{Area}
\DeclareMathOperator{\Lie}{Lie}
\DeclareMathOperator{\crit}{Crit}
\DeclareMathOperator{\lagcor}{LagCor}
\begin{document}
\title{Applications of the theory of Floer to symmetric spaces}

\author{Hanwool Bae}
\address{Center for Quantum Structures in Modules and Spaces, Seoul National University, Seoul, South Korea}
\email{hanwoolb@gmail.com}

\author{Chi Hong Chow}
\address{The Institute of Mathematical Sciences and Department of Mathematics, The Chinese University of Hong Kong, Shatin, Hong Kong}
\email{chchow@math.cuhk.edu.hk}

\author{Naichung Conan Leung}
\address{The Institute of Mathematical Sciences and Department of Mathematics, The Chinese University of Hong Kong, Shatin, Hong Kong}
\email{leung@math.cuhk.edu.hk}

\begin{abstract}
We quantize the problem considered by Bott-Samelson who applied Morse theory to any compact symmetric space $G/K$ and the associated real flag manifold $G_{\mathbb{R}}/B$ which is a real locus of a complex partial flag variety $G_{\mathbb{C}}/P_{\sigma}$. We prove that the Pontryagin ring $H_{-*}(\Omega(G/K))$ of the based loop space $\Omega(G/K)$ is isomorphic to the Floer cohomology ring $HF^*(G_{\mathbb{R}}/B,G_{\mathbb{R}}/B)$ after localization. When $G/K$ is a Lie group, this is a conjecture of Peterson, proved combinatorially by Lam-Shimozono, in the context of quantum cohomologies of complex flag varieties.

Our approach is geometric in nature: we construct a Lagrangian correspondence from $T^*(G/K)$ to $G_{\mathbb{C}}/P_{\sigma}$ which geometrically composes with a cotangent fiber to $G_{\mathbb{R}}/B$, and compute the linear part of the associated Ma'u-Wehrheim-Woodward's $A_{\infty}$ homomorphism from a Floer model of $\Omega(G/K)$ to $CF^*(G_{\mathbb{R}}/B,G_{\mathbb{R}}/B)$. The crux is to make use of the geometry of $G/K$ to construct specific perturbation data which enables us to reduce the computations to the case when $G/K$ is a torus.
\end{abstract}


\maketitle


\section{Introduction}\label{intro}
Let $G_{\RR}$ be a real reductive Lie group. Fix a Cartan involution $\involute$ which gives a decomposition
\[ \gg_{\RR}=\kk\oplus\sqrt{-1}\pp\]
of $\gg_{\RR}:=\Lie(G_{\RR})$ into the $(+1)$-eigenspace $\kk$ and $(-1)$-eigenspace $\sqrt{-1}\pp$. Motivated by the famous paper of Bott-Samelson \cite{BS}, we consider the following two spaces: any \textit{compact symmetric space} $G/K$ associated to the symmetric pair $(\gg,\kk)$ where $\gg:=\kk\oplus\pp$, and the \textit{real complete flag manifold} $G_{\RR}/B$, the quotient of $G_{\RR}$ by its Borel subgroup $B$. The latter is a Lagrangian of a complex flag variety $G_{\CC}/P_{\involute}$ where $G_{\CC}$ is the complexification of $G_{\RR}$ and $P_{\involute}$ is a parabolic subgroup\footnote{The parabolic type of $P_{\involute}$ is defined to be the set of blackened vertices in the Satake diagram associated to $(\gg,\kk)$.}.

\vspace{0.2cm}
\noindent\textbf{Main theorem.} \textit{There exists a ring homomorphism
\[ H_{-*}(\O(G/K))\ra HF^*(G_{\RR}/B,G_{\RR}/B)\]
from the Pontryagin ring of the based loop space $\O(G/K)$ to the Floer cohomology ring of $G_{\RR}/B$. It is injective and becomes an isomorphism after localization for the following cases:
\begin{enumerate}
\item $G/K$ is a Lie group and the coefficient ring is arbitrary;
\item $G/K$ is arbitrary and the coefficient ring is equal to $\ZZ_2$.
\end{enumerate}}

\begin{example} \label{groupcase} Let $H$ be a compact connected semi-simple Lie group. Consider $G_{\RR}:=H_{\CC}$, the complexification of $H$. Then $G=H\times H$ and $K=\Delta_H$ is the diagonal so that $G/K=(H\times H)/\Delta_H\simeq H$. The real flag manifold $G_{\RR}/B$ is equal to the coadjoint orbit $H/T$ ($T$ is a maximal torus in $H$) and sits inside $G_{\CC}/P_{\involute}\simeq H/T^-\times H/T$ as the diagonal. By the famous theorem of Piunikhin-Salamon-Schwarz \cite{PSS}, $HF^*(G_{\RR}/B,G_{\RR}/B)$ is isomorphic to the quantum cohomology ring $QH^*(H/T)[Q^{\vee}]$ where $Q^{\vee}$ is the unit lattice of $T$.

\begin{corollary} $H_{-*}(\O H)$ is isomorphic as rings to $QH^*(H/T)[Q^{\vee}]$ after localization.
\end{corollary}

\noindent This is the simplified version\footnote{The complete version is an explicit description of the ring map in terms of the affine Schubert classes in the source and the quantum Schubert classes in the target.} of a conjecture of Peterson \cite{Peter}. The only known proof, due to Lam-Shimozono \cite{LS}, is combinatorial in nature. See also the work of Leung-Li \cite{LLi1, LLi2, LLi3}. Our approach, as we will explain shortly, provides a geometric perspective which is not given in \cite{LS}.
\end{example}

\begin{example} \label{glnr} Consider $G_{\RR}:=GL(n,\RR)$, the general linear group over $\RR$. A Cartan involution $\involute$ is given by $A\mapsto (A^{-1})^T$. We have $G=U(n)$ and $K=O(n)$ so that $G/K=U(n)/O(n)$, the Lagrangian Grassmannian. A typical Borel subgroup of $G_{\RR}$ is the subgroup of upper triangular matrices so that $G_{\RR}/B$ is equal to the classical complete flag manifold over $\RR$
\[ F\ell(n;\RR):=\{\mathbf{0}\subset V_1\subset\cdots\subset V_n=\RR^n|~\dim V_i=i\}.\]
It is a real locus of the complex flag variety $F\ell(n;\CC)$ defined similarly. Our main theorem implies that $H_{-*}(\O(U(n)/O(n));\ZZ_2)$ is isomorphic to $HF^*(F\ell(n;\RR),F\ell(n;\RR);\ZZ_2)$ after localization.
\end{example}

To construct the ring homomorphism for our main theorem, we start with a natural Lagrangian correspondence
\[ C:T^*(G/K)\ra G_{\CC}/P_{\involute}\]
which geometrically composes with a cotangent fiber $L$ to give $G_{\RR}/B$. It is a special case of a construction of Guillemin-Sternberg \cite{GS} who generalized the notion of \textit{moment correspondence} introduced by Weinstein \cite{W}. Its existence can be understood as follows:\footnote{For simplicity, we assume that $G$ is semi-simple and $G/K$ is of adjoint type. We emphasize that this assumption is not necessary for the main theorem.} The cotangent bundle $T^*(G/K)$ is diffeomorphic to the complex symmetric space $G_{\CC}/K_{\CC}$. It is well-known that the latter has a \textit{wonderful compactification} $\ol{G_{\CC}/K_{\CC}}$ \`a la de Concini-Procesi \cite{wonderful} which is a smooth projective variety containing $G_{\CC}/K_{\CC}$ as a Zariski open subset whose complement is a normal crossing divisor such that $G_{\CC}/P_{\involute}$ is equal to the intersection of all irreducible components of this divisor. In this picture, $C$ corresponds to a torus bundle associated to the normal bundle of $G_{\CC}/P_{\involute}$ in $\ol{G_{\CC}/K_{\CC}}$.

By the theory of pseudoholomorphic quilted surfaces developed by Ma'u-Wehrheim-Woodward \cite{MWW}, we obtain an $A_{\infty}$ homomorphism
\begin{equation}\label{phi}
\Phi_C: CW^*(L,L)\ra CF^*(G_{\RR}/B,G_{\RR}/B).
\end{equation}
See also the work of Evans-Lekili \cite{EL}. Observe that the source of $\Phi_C$ is infinite dimensional while the target is finite dimensional. We make $CF^*(G_{\RR}/B,G_{\RR}/B)$ infinite dimensional as well, by introducing a natural notion of capping disks. See Section \ref{capdisk} for its definition. By a theorem of Abbondandolo-Schwarz \cite{AS} and Abouzaid \cite{Ab_JSG}, the rings $H_{-*}(\O(G/K))$ and $HW^*(L,L)$ are isomorphic. We now rephrase our main theorem in a more precise form.

\begin{theorem} \label{main} Under the same assumptions as in the main theorem, $H^*(\Phi_C)$ is injective, and there exists a multiplicative subset $S$ of $HW^*(L,L)$ consisting of central elements such that $H^*(\Phi_C)$ becomes an isomorphism after localizing $S$.
\end{theorem}

The most technical part of the proof of Theorem \ref{main} is the analysis of the linear part $\Phi^1_C$ of $\Phi_C$. In Section \ref{perturbdataH}, we will construct bases $\{x_q\}_{q\in Q^{\vee}}$ and $\{y_{w,q}\}_{(w,q)\in W\times Q^{\vee}}$ of $CW^*(L,L)$ and $CF^*(G_{\RR}/B,G_{\RR}/B)$ respectively, where $Q^{\vee}$ is the unit lattice of a maximal torus in $G/K$ and $W$ is the Weyl group. Denote by $\langle \Phi^1_C(x_q), y_{w',q'}\rangle$ the entries of the matrix representing $\Phi^1_C$ with respect to these bases. The outcome of the analysis is described in the following

\vspace{0.2cm}
\noindent \textbf{Key lemma.}  \textit{There exist functions $w:Q^{\vee}\ra W$ and $\ell':W\ra \RR$ such that for any $q\in Q^{\vee}$,
\begin{enumerate}
\item $\langle \Phi^1_C(x_q), y_{w(q),q}\rangle = 1$;
\item if $\langle \Phi^1_C(x_q), y_{w',q'}\rangle\ne 0$, then $\ell'(w')<\ell'(w(q))$ unless $(w',q')=(w(q),q)$.
\end{enumerate}
}

\vspace{0.2cm}
Another consequence of our key lemma is
\begin{theorem} \label{maincor} The Pontryagin rings $H_{*}(\O(G/K);\ZZ_2)$ and $H_{*}(\O H;\ZZ)$ are finitely generated.
\end{theorem}

\begin{remark} The group case of Theorem \ref{maincor} was solved completely by Bott \cite{BMich}. In general, the problem about whether or not $H_{*}(\O(G/K);\coeff)$ is finitely generated is non-trivial, since, due to possible existence of divided power structures, an algebra which is additively isomorphic to a finitely generated algebra is not necessarily finitely generated, even over a field (of positive characteristics). For example, consider the cohomology ring of $\O S^3$. We have, by Serre's spectral sequence, $H^*(\O S^3;\coeff)\simeq \bigoplus_{i=0}^{\infty}\coeff \langle x_i\rangle$ for any $\coeff$, where $\deg(x_i)=2i$. The generators satisfy $x_i\cup x_j = \frac{(i+j)!}{i! j!} x_{i+j}$. Take $\coeff=\ZZ_2$. By elementary number theory, the last equation implies that $x_0,\ldots, x_{2^N-1}$ cannot generate $x_{2^N}$ multiplicatively for any natural number $N$. We conclude that $H^*(\O S^3;\ZZ_2)$ is not finitely generated.
\end{remark}

The proof of the key lemma uses an action-functional argument. In typical applications of this argument, constant solutions are those which contribute to the leading terms. In our case, however, we have to refine it by establishing a positive lower bound for the energy and showing that it is attained by a unique non-constant pseudoholomorphic quilt.

To achieve this, we first consider the open dense subset $\UU\subset T^*(G/K)\simeq T(G/K)$ consisting of vectors which are tangent to a unique maximal torus in $G/K$. This uniqueness property allows us to construct a well-defined \Ham torus action on $\UU$ using parallel transport. One can show that $G_{\CC}/P_{\involute}$ is a symplectic quotient of the action (Proposition \ref{comparesymplform}).

Next, we complexify this action with respect to the canonical almost complex structure $J_0$ on $T^*(G/K)$. Notice that complex orbits are not the cotangent bundles of maximal tori but open subsets which are roughly $1/|W|$ of them, where $W$ is the Weyl group. Then the almost complex structures we use for the moduli problems will be small perturbations of $J_0$ which preserve all complex orbits (Definition \ref{condJ}). A relative version of unique continuation theorem is needed in order to achieve the transversality for all the pseudoholomorphic curves involved (Lemma \ref{UCLemma}).

Finally, we construct a retraction contracting $\UU$ onto the coisotropic submanifold which defines the symplectic quotient $G_{\CC}/P_{\involute}$. The aforementioned lower bound is then obtained by applying this retraction to every pseudoholomorphic quilt (Lemma \ref{end1}). Moreover, this lower bound is attained by those quilts which are contained in one of the complex orbits. Since these orbits lie in the cotangent bundles of maximal tori, it is not difficult to show that there is precisely one such quilt with fixed input data, by passing to the universal cover and applying a version of the Riemann mapping theorem (Proposition \ref{B+C/2}).

One technical difficulty is to handle those quilts which do not lie inside $\UU$, preventing one from applying the above retraction directly. This is done by using the fact that the complement $T^*(G/K)\setminus\UU$ is a stratified subset of codimension at least two (Lemma \ref{codim2}). Two is a magic number for Floer theorists: every pseudoholomorphic curve can be perturbed to one which intersects a given codimension-two submanifold along a discrete set. This measure-zero set does not cause any trouble when one analyses the energy of the curve.


\section*{Acknowledgements}
We are grateful to Cheol-Hyun Cho, Otto van Koert, Yanki Lekili, Changzheng Li and Weiwei Wu for useful comments and encouragements. Chi Hong Chow would like to thank Cheuk Yu Mak for explaining him, with great patience, the chapter on transversality in McDuff-Salamon's book \cite{MS}. This work was substantially supported by grants from the Research Grants Council of the Hong Kong Special Administrative Region, China (Project No. CUHK14301619 and CUHK14306720), the National Research Foundation of Korea (NRF) grant funded by the Korea government (MSIT) (No. 2020R1A5A1016126) and a direct grant.

\section{Basic Settings: Lie Theory} \label{Lie}
\subsection{Preliminaries} \label{Lie-basic}
The following materials can be found in \cite{twist, Warner}.

Let $G$ be a compact connected Lie group. Fix an Ad-invariant metric $\langle -,-\rangle$ on its Lie algebra $\gg$ so that $\gg\simeq \gg^{\vee}$ as $G$-modules. Throughout the paper, we will denote the (co)adjoint action simply by $g\cdot X$ ($g\in G$, $X\in \gg$ or $\gg^{\vee}$).

Let $\involute$ be an involutive automorphism of $G$. Let $K$ be a closed subgroup lying between the fixed-point subgroup $G^{\involute}$ of $\involute$ and its identity component $G_0^{\involute}$, i.e.
\[G_0^{\involute}\subseteq K\subseteq G^{\involute}.\]
Consider the derivative $D\involute:\gg\ra\gg$ of $\involute$. We have $\gg=\kk\oplus\pp$ where $\kk$ and $\pp$ are the $(+1)$-eigenspace and $(-1)$-eigenspace of $D\involute$ respectively. Notice that $\kk$ is also the Lie algebra of $K$. It is well-known that
\[ [\kk,\kk]\subseteq \kk,\quad [\kk,\pp]\subseteq \pp,\quad [\pp,\pp]\subseteq \kk.  \]
Let $\t\subseteq\pp$ be an abelian subalgebra of $\gg$ which is contained in $\pp$ and is maximal among all subalgebras of the same kind. We call any such subalgebra a \textit{maximal torus} in $\pp$. Since $\ad$ is skew-symmetric with respect to $\langle -,-\rangle$, we have the following root space decomposition
\begin{equation}\label{rootsp0}
\gg=\cent{\t}{\gg}\oplus \bigoplus_{[\a]\in R_{\t}/_{\pm}}\gg_{\t,[\a]}
\end{equation}
where $\cent{\t}{\gg}$ is the centralizer of $\t$ in $\gg$, $R_{\t}$ is a set of non-zero linear forms $\a:\t\ra\RR$, called \textit{restricted roots}, and $\gg_{\t,[\a]}=\bigcap_{X\in\t}\ker(\ad(X)^2+4\pi^2\a(X)^2\id  )$. By definition, $-R_{\t}=R_{\t}$ so that the above direct sum is taken over the set of representatives of the quotient set $R_{\t}/_{\pm}$. Since $\t\subseteq \pp$, $D\involute$ preserves each summand of \eqref{rootsp0}, and hence they further split into $(\pm 1)$-eigenspaces:
\begin{align}
\cent{\t}{\gg} & =\cent{\t}{\kk}\oplus\cent{\t}{\pp} \nonumber \\
\gg_{\t,[\a]} &= \kk_{\t,[\a]}\oplus \pp_{\t,[\a]} \nonumber.
\end{align}
But $\t$ is a maximal torus in $\pp$ so that $\cent{\t}{\pp}=\t$, and hence we have
\begin{align}\label{rootsp1}
\kk &= \cent{\t}{\kk}\oplus \bigoplus_{[\a]\in R_{\t}/_{\pm}} \kk_{\t,[\a]} \nonumber \\
\pp &= \t \oplus \bigoplus_{[\a]\in R_{\t}/_{\pm}} \pp_{\t,[\a]} .
\end{align}

Notice that for any $X\in\t$, $\ad(X)$ maps $\kk_{\t,[\a]}$ into $\pp_{\t,[\a]}$ and vice versa, and it interchanges them isomorphically if $\a(X)\ne 0$. It follows that $\kk_{\t,[\a]}$ and $\pp_{\t,[\a]}$ have the same dimension which we call the \textit{multiplicity} of $\a$ and denote by $\ma$. Moreover, any $X\in\t$ with $\a(X)\ne 0$ for all $\a$ is contained in a unique maximal torus in $\pp$. For our purpose, we usually reverse the order by starting with any point in $\pp$ with the above property. These points are abundant. To see this, we first recall the following well-known
\begin{lemma} \label{Liewk1} $K$ acts transitively on the set of maximal tori in $\pp$, and $\pp=\bigcup_{k\in K}k\cdot\t$. \hfill$\square$
\end{lemma}

\begin{definition} \label{goodopen} Define $U$ to be the set of $X\in \pp$ which lies in a unique maximal torus $\t_X$ in $\pp$.
\end{definition}

\begin{lemma} \label{Liewk1.5} The subset $U$ is open in $\pp$ and is invariant under the $K$-action. Its complement $\pp\setminus U$ is a stratified subset with finitely many strata each of which has codimension at least two.
\end{lemma}
\begin{proof} It is clear that $U$ is $K$-invariant. The openness of $U$ follows from the equality
\[\pp=\t\oplus [\kk,\t] \]
which is an immediate consequence of \eqref{rootsp1}. To show the last part, we fix a maximal torus $\t$ in $\pp$. By Lemma \ref{Liewk1}, $\pp\setminus U$ is equal to $K\cdot S$ where $S:=\bigcup_{\a\in R_{\t}}\ker(\a)$. Since $S$ has codimension one in $\t$, it suffices to show that for any $X\in S$, $[\kk,X]\subsetneq [\kk,\t]$. This follows from \eqref{rootsp1} and the observation that $\ad(X)$ vanishes on $\kk_{\t,[\a]}$ for at least one $\a$.
\end{proof}

\begin{remark}\label{Lie-remark} The codimension of $\pp\setminus U$ is equal to two precisely when there exists a restricted root $\a$ such that $\ma =1$ and $\frac{1}{2}\a$ is not a restricted root.
\end{remark}

Let $X\in U$. Then it determines a set $\Ra{X}$ of representatives of the quotient $R_{\t_X}/_{\pm}$ defined by
\[\Ra{X}:=\{\a\in R_{\t_X}|~\a(X)>0\}.  \]
From now on, we use $\gg_{X,\a}$ (resp. $\ka{X}$ resp. $\pa{X}$) in place of $\gg_{\t_X,[\a]}$ (resp. $\kk_{\t_X,[\a]}$ resp. $\pp_{\t_X,[\a]}$) so that \eqref{rootsp1} becomes
\begin{align} \label{rootsp}
\kk &= \cent{\t_X}{\kk}\oplus \bigoplus_{\a\in \Ra{X}} \ka{X} \nonumber \\
\pp &= \t_X \oplus \bigoplus_{\a\in \Ra{X}} \pa{X}.
\end{align}
We will need two more notions: the \textit{Weyl group} $W_X$, the subgroup of $\Aut(\t_X)$ generated by the reflections across the walls $\ker(\a)$, $a\in \Ra{X}$, and the \textit{Weyl chamber} $\wc_X$, the connected component of $\t\setminus \bigcup_{\a\in \Ra{X}}\ker(\a)$ which contains $X$. It is well-known that $W_X$ acts simply transitively on the set of Weyl chambers in $\t_X$. The following lemma is well-known and will be useful later.
\begin{lemma} \label{Liewk2} Let $X\in U$.
\begin{enumerate}
\item The action on $\t_X$ by the normalizer $\norm{\t_X}{K}$ of $\t_X$ in $K$ induces an isomorphism
\[\norm{\t_X}{K}/\cent{\t_X}{K}\simeq W_X.  \]
\item We have
\[ \cent{X}{G}=\exp(\t_X)\cdot \cent{X}{K} \]
where $\exp:\gg\ra G$ is the exponential map.  \hfill$\square$
\end{enumerate}
\end{lemma}

\begin{remark} Lemma \ref{Liewk2}(1) does not require $K$ to be connected. The key point is that if $k\in K$ normalizes $\t_X$, then it normalizes $K_{\t_X}:=\cent{\t_X}{K}$. Conjugating $k$ with an element of $K_{\t_X}$ if necessary, $k$ normalizes a maximal torus $\t$ in $\Lie(K_{\t_X})$. It follows that it normalizes the maximal torus $\t':=\t\oplus\t_X$ in $\gg$, and hence induces a Weyl group element of $W(\t',\gg)$ which commutes with $D\involute$. It is well-known that such an element restricts to an element of $W_X$.
\end{remark}
\subsection{Coadjoint orbits and real flag manifolds} \label{Lie-flag}
We continue to use the notations defined in the previous subsection. Fix $X_0\in U$.
\begin{definition} Define
\begin{align*}
\CO&:= G\cdot X_0 \subset \gg^{\vee}\\
\RF&:= K\cdot X_0 \subset \pp.
\end{align*}
They are a coadjoint orbit and a real flag manifold respectively.
\end{definition}

\begin{example} Let $H$ be a compact connected Lie group. Define $\involute$ to be the involution on $G:=H\times H$ which interchanges the two components. Then $\kk$ (resp. $\pp$) is equal to the (resp. anti-) diagonal of $\gg=\mathfrak{h}\oplus \mathfrak{h}$. One sees easily that $\CO=H/T\times H/T$ and $\RF$ is equal to the diagonal of $\CO$, where $T$ is a maximal torus in $H$.
\end{example}

There is a well-known symplectic form $\w_{\CO}$ on $\CO$ defined by
\begin{equation}\nonumber
\w_{\CO}([\eta,X],[\zeta,X]):= -\langle [\eta,\zeta] , X \rangle
\end{equation}
for any $X\in\CO$ and $\eta,\zeta\in \gg$.

\begin{lemma}\label{orientable} $\RF$ is a connected, orientable Lagrangian of $(\CO,\w_{\CO})$.
\end{lemma}
\begin{proof}
The connectedness follows from Lemma \ref{Liewk2}(1): $K\cdot X_0$ and $K_0\cdot X_0$ have the same intersection with $\t_{X_0}$. To prove that $\RF$ is a Lagrangian, notice that $-D\involute:\gg\ra\gg$ restricts to an anti-symplectic involution on $\CO$ such that $\RF$ is a connected component of its fixed-point locus. To show that $\RF$ is orientable, it suffices to show that its normal bundle in $\pp$ is orientable. This bundle is in fact trivial, and the proof uses Lemma \ref{Liewk2}(1) and the fact that $W_{X_0}$ acts simply transitively on the set of Weyl chambers. See \cite{BSZ} for another proof of the orientability.
\end{proof}

\subsection{Symmetric spaces} \label{Lie-sym}
Let $G$ and $K$ be as before. We consider the quotient $G/K$, called a compact symmetric space. The $K$-invariant metric on $\pp$ induces a Riemannian metric on $G/K$ whose associated Levi-Civita connection $\nabla$ is described as follows. Consider the left-invariant Maurer-Cartan form $\b$ on $G$, i.e. the unique left-invariant $\gg$-valued 1-form on $G$ which is equal to the identity at $e\in G$. It satisfies the Maurer-Cartan equation
\begin{equation}\label{MC}
d\b=-\frac{1}{2}[\b,\b].
\end{equation}
Denote by $\pr_{\kk}:\gg\ra\kk$ and $\pr_{\pp}:\gg\ra\pp$ the orthogonal projections. Then $\pr_{\kk}\circ\b$ is a connection 1-form on the principal $K$-bundle $G\ra G/K$. It is proved in \cite{twist} that the associated connection on $T(G/K)=G\times_K\pp$ preserves the metric and is torsion-free, and hence it coincides with $\nabla$. Using this fact, it is not hard to deduce the following properties concerning $\nabla$.
\begin{lemma} \label{symwk1}$~$
\begin{enumerate}
\item Consider the canonical splitting $T(T(G/K))=\Hor\oplus \Ver$ induced by $\nabla$, where $\Hor$ and $\Ver$ are the horizontal and vertical subbundles of $T(T(G/K))$ respectively. At any point $[g:Y]\in T(G/K)=G\times_K\pp$,
\[ \Hor=\{ \pi_*(g_*\eta,0)|~\eta\in\pp \}\]
where $\pi:G\times \pp\ra G\times_K\pp$ is the quotient map.
\item Put $p_0:=eK$. Identify $T_{p_0}G/K$ with $\pp$. The Riemannian curvature tensor $R_{p_0}$ at $p_0$ satisfies
\[ R_{p_0}(X,Y)Z=-[[X,Y],Z]\]
for any $X,Y,Z\in\pp$. \hfill$\square$
\end{enumerate}
\end{lemma}

We will make use of a certain class of submanifolds of $G/K$. Recall we have fixed $X_0\in\pp$ in the previous subsection. Put $\t_0:=\t_{X_0}$ and $T_0:=\exp(\t_0)K$. Then $T_0$ is a torus embedded in $G/K$ and $\t_0$ is its universal cover. Define
\[ Q^{\vee} :=\{ q\in \t_0|~ \exp(q)\in K\}\]
to be the unit lattice of the universal covering $\t_0\ra T_0$. It is clear that $Q^{\vee}=\frac{1}{2}\exp^{-1}(e)\cap \t_0$.

\begin{lemma} \label{flattorus}  The torus $T_0$ is a totally geodesic submanifold of $G/K$ with flat tangential and normal curvatures.
\end{lemma}
\begin{proof}
The first assertion follows from the observation that the geodesic symmetry $gK\mapsto \involute(g)K$ preserves $T_0$. The rest follows from the first assertion and Lemma \ref{symwk1}(2).
\end{proof}

\begin{definition} We call any submanifold $g\cdot T_0$, $g\in G$. a \textit{maximal torus} in $G/K$.
\end{definition}

Now, we consider the cotangent bundle $M:=T^*(G/K)$. It is an exact symplectic manifold with symplectic form $\w_M=d\l_M$ where $\l_M:=\mb{p}d\mb{q}$ is the standard Liouville form. We identify $M$ with $T(G/K)$ using the metric so that we have the corresponding splitting $TM=\Hor\oplus\Ver$.

Recall $U$ which was defined in Section \ref{Lie-basic}. Since it is $K$-invariant (Lemma \ref{Liewk1.5}), it makes sense to define $\UU:=G\times_K U\subset G\times_K\pp^{\vee}=M$. Lemma \ref{Liewk1.5} implies the following
\begin{lemma} \label{codim2} The subset $\UU$ is open in $M$ and is invariant under the $G$-action. Let $L$ denote the cotangent fiber at any point in $G/K$. The complements $M\setminus \UU$ and $L\setminus \UU$ are stratified subsets with finitely many strata each of which has codimension at least two (in $M$ and $L$ respectively). \hfill$\square$
\end{lemma}

We end this section by deriving a global analogue of decomposition \eqref{rootsp}. First, we identify, for any $Y\in U$, $\Ra{Y}$ with $\Ra{X_0}$ in the following way. Choose $k\in K$ such that $\wc_Y=k\cdot \wc_{X_0}$ which is possible by Lemma \ref{Liewk1} and Lemma \ref{Liewk2}(1). The restriction $\Ad(k)|_{\t_0}:\t_0\ra \t_Y$ gives the desired identification $\Ra{Y}\simeq \Ra{X_0}$ which does not depend on the choice of $k$, thanks to Lemma \ref{Liewk2}(1). Denote by $\pi:G\times \pp^{\vee}\ra M$ the quotient map.
\begin{definition}\label{Ua} Define distributions $\UU_0$ and $\UU_{\a}$, $\a\in\Ra{X_0}$ on $\UU$ as follows.
\begin{enumerate}
\item $\UU_0:=\UU_0^h\oplus \UU_0^v$ where at a point $[g:Y]\in \UU$,
\begin{align*}
\UU_0^h&:= \{ \pi_*(g_*\eta,0)|~\eta\in\t_Y\} \subset \Hor;\\
\UU_0^v&:= \{ \pi_*(0,\zeta )|~\zeta\in\t_Y\} \subset \Ver.
\end{align*}
\item $\UU_{\a}:=\UU_{\a}^h\oplus \UU_{\a}^v$ where at a point $[g:Y]\in \UU$,
\begin{align*}
\UU_{\a}^h&:= \{ \pi_*(g_*\eta,0)|~\eta\in\pa{Y}\} \subset \Hor;\\
\UU_{\a}^v&:= \{ \pi_*(0,\zeta )|~\zeta\in\pa{Y}\} \subset \Ver.
\end{align*}
\end{enumerate}
\end{definition}

\begin{lemma} \label{rootspglobal} $\UU_0$ and $\UU_{\a}$ are symplectic subbundles of $T\UU$. Moreover, we have
\begin{equation} \label{rootspglobaleq}
T\UU=\UU_0\oplus\bigoplus_{\a\in\Ra{X_0}}\UU_{\a}
\end{equation}
where the direct summands are pairwise symplectic orthogonal to each other.
\end{lemma}
\begin{proof}
We use Lemma \ref{symwk1}(1) to express vectors in the horizontal subbundle $\Hor\subset TM$. Recall $\l_M:=\mb{p}d\mb{q}$ is the standard Liouville form. We have
\[\pi^*\l_M=\langle \b,Y\rangle \]
which gives, by the Maurer-Cartan equation \eqref{MC},
\begin{equation}\label{sympl}
\pi^*\w_M=-\frac{1}{2}\langle [\b,\b], Y\rangle -\langle\b,dY\rangle.
\end{equation}
Given $\a_1,\a_2\in\Ra{X_0}\cup\{0\}$. For $i=1,2$, let $\eta_i, \zeta_i\in\pp_{Y,\a_i}$ if $\a_i\ne 0$ and $\in \t_Y$ if $\a_i=0$. (Recall we have identified $\Ra{Y}$ with $\Ra{X_0}$.) Then, by \eqref{sympl},
\[ \pi^*\w_M((g_*\eta_1,\zeta_1),(g_*\eta_2,\zeta_2))=-\langle [\eta_1,\eta_2],Y\rangle - (  \langle \eta_1,\zeta_2 \rangle -  \langle \eta_2,\zeta_1 \rangle) .\]
The first term in the last expression vanishes because $[\eta_1,\eta_2]\in\kk$. The rest is obvious.
\end{proof}

\section{Basic Settings: Floer Theory} \label{Floer}
The main references for quilted Floer theory are \cite{EL, MWW}. We also recommend \cite{WW1, WW1.5, WW2, WWorient, WW3}. For wrapped Floer theory, we refer to \cite{Ab_JSG}.
\subsection{A Lagrangian correspondence} \label{Lagcorr}
Recall we have put $M=T^*(G/K)$, the cotangent bundle of a compact symmetric space $G/K$ and $\CO=G\cdot X_0$, the coadjoint orbit through a fixed point $X_0$. They are given the symplectic forms $\w_M$ and $\w_{\CO}$ respectively. We are going to construct a Lagrangian correspondence
\[  C:(M,\w_M)\ra (\CO,\w_{\CO}) \]
by expressing it as the graph of a symplectic quotient. Define $\C$ to be the submanifold
\[\C :=G\times_K\RF\subset G\times_K\pp^{\vee}\simeq T^*(G/K).\]
Recall $U\subset\pp$ from Definition \ref{goodopen}. Given any $Y\in U$. By definition, there exists a unique maximal torus $\t_Y$ in $\pp$ which contains $Y$. By Lemma \ref{Liewk2}, $\RF$ intersects the Weyl chamber $\wc_Y$ at a unique point which we denote by $\phi_0(Y)$. Clearly, $\phi_0$ is $K$-equivariant so that the map
\[
\begin{array}{rccc}
\phi: &\UU&\ra& \C\\
&[g:Y] &\mapsto & [g:\phi_0(Y)]
\end{array}
 \]
is well-defined. We also define $\sq: \C \ra \CO$ by  $\sq([g:Y]):= g\cdot Y$.

\begin{proposition} \label{comparesymplform}$~$
\begin{enumerate}
\item For any point $g\cdot X_0\in \CO$, the fiber $(\pj)^{-1}(g\cdot X_0)$ is equal to
\begin{equation}\label{fiber}
\{ [h:Y]|~h\in g\cdot\exp(\t_0), Y\in \wc_{X_0}\}
\end{equation}
which is a non-compact symplectic submanifold of $M$. These fibers exhibit the leaves of the distribution $\UU_0$.
\item Let $\a,\a_1,\a_2\in \Ra{Y}\simeq \Ra{X_0}$ with $\a_1\ne \a_2$. We have
\begin{equation}\label{aequal}
 (\pj)^*\w_{\CO}|_{\UU_{\a}^{\otimes 2}}= \frac{\a(\phi_0(Y))}{\a(Y)} \w_M|_{\UU_{\a}^{\otimes 2}}
\end{equation}
and
\begin{equation}\label{adiff}
 (\pj)^*\w_{\CO}|_{\UU_{\a_1}\otimes\UU_{\a_2}}= \w_M|_{\UU_{\a_1}\otimes\UU_{\a_2}} = 0.
\end{equation}
\item $\C$ is a coisotropic submanifold of $M$ and $\sq$ is the symplectic quotient map.
\end{enumerate}
\end{proposition}
\begin{proof}
\begin{enumerate}
\item[]
\item Since $\UU_0$ is $G$-invariant and $\phi, \sq$ are $G$-equivariant, it suffices to consider the fiber of $\pj$ at $X_0$. By definition, it is equal to
\[ \{[h:Y]|~Y\in U, h\cdot \phi_0(Y)=X_0\}.\]
Let $[h:Y]$ be an element of this set. Pick $k\in K$ such that $\phi_0(Y)=k\cdot X_0$. Then $hk\in \cent{X_0}{G}=\exp(\t_0)\cdot\cent{X_0}{K}$ (Lemma \ref{Liewk2}) which implies that $[h:Y]$ is an element of the set \eqref{fiber}, by a change of its representative. The rest is obvious.

\item Let $\a\in \Ra{Y}$. By Lemma \ref{symwk1}, $\UU_{\a}$ consists of vectors of the form $\pi_*(g_*\eta,\zeta)$ where $\eta,\zeta\in \pa{Y}$ and $\pi:G\times \pp\ra G\times_K\pp$ is the quotient map. For $i=1,2$, let $\a_i\in \Ra{Y}$ and $\xi_i:= (g_*\eta_i,\zeta_i)$ with $\eta_i,\zeta_i\in \pp_{Y,\a_i}$. Put $\zeta'_i:=\frac{1}{4\pi^2\a_i(Y)^2}[Y,\zeta_i]$. Then $\zeta'_i\in\kk_{Y,\a_i}$ and $[\zeta'_i,Y]=\zeta_i$. We have
\begin{align*}
 D(\pj\circ\pi)_{[g:Y]}\xi_i &= \left.\frac{d}{dt}\left( ge^{t\eta_i}e^{t\zeta'_i}\cdot \phi_0(Y) \right)\right|_{t=0}\\
 & = [g\cdot (\eta_i+\zeta'_i),g\cdot \phi_0(Y)].
\end{align*}
It follows that
\begin{align*}
(\pj\circ\pi)^*\w_{\CO}(\xi_1,\xi_2) &= - \langle g\cdot [\eta_1+\zeta'_1, \eta_2+\zeta'_2], g\cdot \phi_0(Y) \rangle \\
&= - \langle [\eta_1+\zeta'_1, \eta_2+\zeta'_2],  \phi_0(Y) \rangle \\
&= - \langle [\zeta'_1,\eta_2] +[\eta_1,\zeta'_2],  \phi_0(Y) \rangle \\
& = \frac{\a_1(\phi_0(Y))}{\a_1(Y)} \langle \eta_2,\zeta_1\rangle-\frac{\a_2(\phi_0(Y))}{\a_2(Y)} \langle \eta_1,\zeta_2\rangle  .
\end{align*}

\noindent On the other hand, by the computation in the proof of Lemma \ref{rootspglobal},
\[\pi^*\w_M(\xi_1,\xi_2) = \langle \eta_2,\zeta_1\rangle - \langle \eta_1,\zeta_2\rangle.\]
If $\a_1=\a_2$, we obtain \eqref{aequal}; if $\a_1\ne\a_2$, we obtain \eqref{adiff}, since $\pp_{Y,\a_1}$ is orthogonal to $\pp_{Y,\a_2}$.

\item It is not hard to see that, for any $[g:Y]\in\C$,
\[T_{[g:Y]}\C = \UU_0^h\oplus \bigoplus_{\a\in\Ra{X_0}}\UU_{\a}.\]
The result then follows from Lemma \ref{rootspglobal}, parts (1) and (2) above.
\end{enumerate}
\end{proof}

\begin{definition} Define $C\subset M^-\times \CO$ to be graph of the symplectic quotient $\sq:\C\ra \CO$, i.e.
\[C:=\{ ([g:Y],X)\in M\times \CO|~Y\in\RF, g\cdot Y=X\}.\]
\end{definition}

\begin{remark}
\begin{enumerate}
\item[]
\item If $G/K=(H\times H)/H \simeq H$ is a compact Lie group, the corresponding Lagrangian correspondence $C$ is equal to the famous moment correspondence \cite{W} associated to the \Ham $H$-manifold $H/T$ where $T$ is a maximal torus in $H$.
\item The symplectic quotient $\sq$ is in fact induced by a \Ham torus action on $\UU$ which is characterized by the property that it restricts to the parallel transport map (with respect to the flat metric) over $\UU\cap T^*T$ for every maximal torus $T$ in $G/K$. This action is used to construct the wonderful compactification \cite{wonderful} of $T^*(G/K)$\footnote{Under the mild assumption that $G/K$ is of adjoint type.} via non-abelian symplectic cuttings \cite{NASC}, where $\CO$ corresponds to the intersection of all boundary divisors.
\item The Lagrangian correspondence $C$ is a special case of a more general construction given in \cite{GS}: let $\G$ be a connected Lagrangian of a symplectic manifold $N$, and $B$ a smooth manifold. Suppose there is a smooth map $f:B\ra Ham(N)\cdot \G$ from $B$ to the orbit through $\G$ under the action by the group of \Ham diffeomorphisms of $N$. Then, provided an obstruction class in $ H^2(B;\RR)$ vanishes, there exists a Lagrangian correspondence $C_f:T^*B\ra N$ whose geometric composition with every cotangent fiber $T^*_bB$ is embedded and equal to $f(b)\cdot\G$.
\end{enumerate}
\end{remark}

For any $p\in G/K$, we put $L_p:=T^*_p(G/K)$ and $\RFF{p}:=g\cdot \RF$ where $p=gK$\footnote{The latter is well-defined since $K$ preserves $\RF$.}.
\begin{lemma} The geometric composition $L_p\circ C$ is embedded and equal to $\RFF{p}$. \hfill{$\square$}
\end{lemma}

In order to apply Floer theory, we need $\CO$, $\RF$ and $C$ to be monotone.
\begin{assume} \label{monotone} There exists $\monoconst>0$ such that the following equality holds
\begin{equation}\label{sumofroot}
\langle - , X_0\rangle|_{\t_0} = 2\monoconst\sum_{\a\in\Ra{X_0}} \ma \a(-).
\end{equation}
\end{assume}
\begin{remark} In order for Assumption \ref{monotone} to make sense, we have to show that if $X_0'\in \t_0$ satisfies \eqref{sumofroot} (with $\langle - , X_0\rangle$ replaced by $\langle - , X_0'\rangle$), then $X_0'\in\wc_{X_0}$. This is a standard fact, at least in the group case, which is proved as follows. Put $\rho:=\sum_{\a\in\Ra{X_0}}\ma \a$, regarded as an element of $\t_0$ by the metric. Let $\b\in\Ra{X_0}$ be a simple root, i.e. a root such that $\frac{1}{2}\b\not\in \Ra{X_0}$ and it defines a boundary wall of $\wc_{X_0}$. Denote by $s_{\b}$ the reflection across $\ker(\b)$. Since the segment joining $X_0$ and $s_{\b}(X_0)$ intersects only one wall, namely $\ker(\b)$, it follows that $s_{\b}$ preserves $\Ra{X_0}\setminus \{\b,2\b\}$ or  $\Ra{X_0}\setminus \{\b\}$, depending on whether $2\b$ is a restricted root or not, and hence
\[s_{\b}(\rho)=\left\{ \begin{array}{ll}
\rho -2(\maaa{\b}+2\maaa{2\b})\b&, 2\b\in\Ra{X_0}\\ [-1em]\\
\rho -2\maaa{\b}\b &, 2\b\not\in\Ra{X_0}
 \end{array}\right. .\]
In both cases, we have $\langle \rho,\b\rangle >0$ so that $\rho\in\wc_{X_0}$ as desired.
\end{remark}

The following proposition, which will be proved in Appendix \ref{2lemma}, is necessary in order to have a well-defined quilted Floer theory for our settings. It implies in particular the monotonicity of the symplectic manifold $(\CO,\w_{\CO})$ as well as the Lagrangians $\RFF{p}$ and $C$.
\begin{proposition}\label{admiss} Let $\ul{\S}$ be a compact genus zero\footnote{Since the Hurewicz map $\pi_2(\CO)\ra H_2(\CO;\ZZ)$ is bijective, the proposition also holds for $\ul{\S}$ of arbitrary genus.} quilted Riemann surface. Label each patch of $\ul{\S}$ by either $M$ or $\CO$, each boundary component by either $L_p$ or $\RFF{p}$ (different boundary components may correspond to different $p$), and each seam by $C$. Let $\ul{u}$ be a quilted map defined on $\ul{\S}$ which respects the labelling. Denote by $\area(\ul{u})$ the sum of the symplectic areas of all patches of $\ul{u}$ and $\mu(\ul{u})$ its Maslov index. Under Assumption \ref{monotone}, we have
\[ \area(\ul{u})=\monoconst\cdot \mu(\ul{u}). \]
\end{proposition}

\begin{corollary} The minimal Maslov numbers of $\RFF{p}$ and $C$ are at least two.
\end{corollary}
\begin{proof}
By Proposition \ref{admiss}, it suffices to show that they are orientable. This follows from Lemma \ref{orientable} and the fact that $C$ fibers over simply-connected $\CO$ with torus fibers which are orientable.
\end{proof}
\subsection{A notion of capping disks and the Floer cochain complexes} \label{capdisk}
Let $p\in G/K$. By applying the machinery developed in \cite{EL} and \cite{MWW} to the Lagrangian correspondence $C$, we obtain an $A_{\infty}$ homomorphism
\[  \Phi_C: CW_b^*(L_p,L_p)\ra CF^*((L_p,C),(L_p,C)).\]
Recall we have to enlarge the target of $\Phi_C$ in order to obtain Theorem \ref{main}. To this end, we introduce a notion of capping disks and describe how $\Phi_C$ and the $A_{\infty}$ structure $\{\mu_k\}_{k=1}^{\infty}$ on its target are modified. Since more than one cotangent fiber will be considered and we ultimately need $(L_p,C)\simeq \RFF{p}$, it is better to carry out the task in the settings of $A_{\infty}$ categories. Let $p_0,p_1\in G/K$ be two points. Put $L_i:=L_{p_i}$ and $\RFF{i}:=\RFF{p_i}$. We construct an $A_{\infty}$ category $\AA$ as follows.
\begin{definition} The objects of $\AA$ are
\[(L_0,C), (L_1,C), \RFF{0}, \RFF{1}.\]
\end{definition}

Next we define the morphism spaces of $\AA$. Observe that every generalized \Ham chord in our settings consists of up to three \Ham chords, each in either $M$ or $\CO$. For example, the generalized \Ham chords for defining $CF^*((L_1,C),\RFF{0})$ are of the form $(x_1,x_2)$ where $x_1$ is a time-$\delta$ \Ham chord in $M$ (for a fixed $\delta>0$) and $x_2$ is a time-1 \Ham chord in $\CO$. They are required to satisfy
\begin{equation}\label{genHam}
x_1(0)\in L_1,~(x_1(\delta),x_2(0))\in C,~ x_2(1)\in\RFF{0}.
\end{equation}
To simplify our notations, we denote any generalized \Ham chord by $\ul{x}=(x_-,x_{\CO},x_+)$ where $x_{\CO}$ is a time-1 \Ham chord in $\CO$, and each of $x_-$ and $x_+$ is either a time-0 (i.e. constant) or a time-$\delta$ \Ham chord in $M$. These (possibly constant) \Ham chords are required to satisfy a condition analogous to \eqref{genHam}. In the above example, $x_-$ is equal to $x_1$, and $x_{\CO}$ is equal to $x_2$. As for $x_+$, recall that $\RFF{0}=L_0\circ C$ is embedded so there is a unique point $[g:Y]\in M$ such that $([g:Y],x_{\CO}(1))\in C$. We have $x_+\equiv [g:Y]$. Denote by $\DH$ the unit upper half-disk and by $\pt_a\DH$ its boundary arc $\pt\DH\setminus(-1,1)$.
\begin{definition} \label{mcd} Let $\ul{x}=(x_-,x_{\CO},x_+)$ be a generalized \Ham chord.
\begin{enumerate}
\item A \textit{\capping disk} for $\ul{x}$ is a pair $(u,\g)$ consisting of maps
\[u:\DH\ra \CO\quad\text{and}\quad\g:\pt_a\DH\ra \C\]
which satisfies
\begin{enumerate}
\item $\g$ is a lift of $u|_{\pt_a\DH}$ with respect to $\sq$;
\item $x_-\#\g\#x_+$ is continuous; and
\item $u(t)=x_{\CO}\left(\frac{1-t}{2}\right)$ for any $t\in [-1,1]$.
\end{enumerate}
\item A \textit{homotopy of \capping disks} for $\ul{x}$ consists of homotopies $s\mapsto u_s$ and $s\mapsto \g_s$ such that for any $s$, $(u_s,\g_s)$ is a \capping disk for $\ul{x}$.
\end{enumerate}
\end{definition}

Let $\ul{L}$ and $\ul{L}'$ be any objects of $\AA$. Denote by $\X(\ul{L},\ul{L}')$ the set of generalized \Ham chords for the cyclic set $(\ul{L},(\ul{L}')^T)$ of Lagrangian correspondences. Fix a coefficient ring $\coeff$.
\begin{definition} \label{morphismspace} The morphism space from $\ul{L}$ to $\ul{L}'$ is defined to be
\[ Hom_{\AA}^*(\ul{L},\ul{L}') := \bigoplus \coeff\langle \ul{x},[u,\g]\rangle \]
where the direct sum is taken over all $\ul{x}\in\X(\ul{L},\ul{L}')$ and all homotopy classes $[u,\g]$ of \capping disks for $\ul{x}$.
\end{definition}

\begin{remark} Strictly speaking, the notation $\coeff\langle \ul{x},[u,\g]\rangle$ which appears in Definition \ref{morphismspace} should be replaced by the determinant line of certain Cauchy-Riemann operator associated to the pair $(\ul{x},[u,\g])$. This point will be clarified in Appendix \ref{sign}. Since sign issue will not arise in the proof of our main theorem, we continue to use the present notation.
\end{remark}

We now describe how to adapt the definitions of $\mu_k$ and $\Phi_C$ to our settings. For simplicity, we call the \capping disks for any input generalized \Ham chord which appears in $\mu_k$ input \capping disks. Let $\ul{u}$ be a pseudoholomorphic quilted map, defined on a quilted surface $\ul{\S}$, which contributes to a term in either $\mu_k$ or $\Phi_C$. This map and the input \capping disks (if there are any) determine a \capping disk for the output \Ham chord which is defined as follows. First notice that $\ul{\S}$ has a unique patch $\S_{\CO}$ which is labelled by $\CO$. Attaching the input \capping disks to the restriction $\ul{u}|_{\S_{\CO}}$ gives, after a reparametrization of the domain, a map $u_{out}:\DH\ra\CO$ which will be one part of the desired \capping disk. It remains to determine the other part, namely a lift $\g_{out}$ of $u_{out}|_{\pt_a\DH}$. Let $\ell$ be a boundary arc of $\S_{\CO}$. Then $\ell$ is labelled by either $C$ or $\RFF{i}$. In the first case, $\S_{\CO}$ has an adjacent patch of which $\ell$ is the common boundary. The restriction of $\ul{u}$ to this patch gives us a lift of $u_{out}|_{\ell}$. In the second case, $u_{out}|_{\ell}$ has a unique lift in $\C$ which lies in $L_i$. Then $\g_{out}$ is defined by concatenating these lifts, alternatively, with the lifts which are part of the input \capping disks.

In order to prove our main theorem, we also need another model for the Floer cohomology of $\RF$, namely the Biran-Cornea \textit{pearl complex} \cite{BCP}. It is the cochain complex freely generated by the critical points of a Morse function $h:\RF\ra\RR$ whose differential counts \textit{pearly trajectories}. As above, we define \capping disks for any critical points so that we obtain the enlarged pearl complex which we denote by $QC^*(\RF)$. One can define a unital associative product $\star$ on its cohomology $QH^*(\RF)$ by counting \textit{pearly triangles}. See \cite{BCP} for more details.

Finally, we remark that the cochain complexes $Hom_{\AA}^*(\ul{L},\ul{L}')$ and $QC^*(\RF)$ are pairwise quasi-isomorphic. For our purpose, we only need the following two quasi-isomorphisms
\[ Y: Hom_{\AA}^*((L_0,C),(L_1,C))\ra Hom_{\AA}^*(\RFF{0},\RFF{1})\]
and
\[ PSS: Hom_{\AA}^*(\RFF{0},\RFF{0})\ra QC^*(\RF).\]
The first map $Y$, constructed in \cite{LL}, is the composition
\[ Hom_{\AA}^*((L_0,C),(L_1,C))\xrightarrow{Y_2} Hom_{\AA}^*((L_0,C),\RFF{1})\xrightarrow{Y_1^{-1}}  Hom_{\AA}^*(\RFF{0},\RFF{1})\]
where $Y_1$ and $Y_2$ are defined by counting configurations shown in Figure 1 and Figure 2 respectively. (They are chain isomorphisms, see Lemma \ref{filY}.)
\begin{center}
\begin{minipage}{7.5cm}
\begin{center}
\vspace{-.1cm}
\begin{tikzpicture}
\tikzmath{\x1 = .6; \x2 = 1.6; \x3=5; \x4=2.5; \x5=1;  \x9=0.45;}

\draw [line width=\x9mm] (\x3-\x4,0.5*\x1) -- (\x3,0.5*\x1);
\draw [line width=\x9mm] (0,1.5*\x1+\x2) -- (\x3,1.5*\x1+\x2);

\draw [->, line width=\x9mm] (0,0) -- (0,\x1);
\draw [->, line width=\x9mm] (0,\x1) -- (0,1.5*\x1+\x2);
\draw [->, line width=\x9mm] (\x3,0.5*\x1) -- (\x3,1.5*\x1+\x2);

\draw [line width=\x9mm, pattern=north east lines, pattern color=gray]  (0, \x1) -- (\x3-\x4-\x5,\x1) to[out=-2,in=150] (\x3-\x4,0.5*\x1) to[out=210,in=2]  (\x3-\x4-\x5,0) -- (0, 0) -- cycle;
\node[anchor =center] at (0,3*\x1+\x2) {};
\node[anchor = west] at (0.45*\x3,\x1+0.5*\x2) {$\CO$};
\node[anchor = south] at (\x3-0.8*\x3, 0.95*\x1) {$C$};
\node[anchor = north] at (\x3-0.8*\x3, 0) {$L_0$};
\node[anchor = south] at (\x3-0.2*\x3, 1.4*\x1+\x2) {$\RFF{1}$};
\node[anchor = north] at (\x3-0.2*\x3, 0.5*\x1) {$\RFF{0}$};
\node[anchor = center] at (\x3-0.8*\x3, 0.5*\x1) {$M$};
\end{tikzpicture}

\vspace{0cm}
\noindent \hypertarget{fig1}{FIGURE 1.}
\vspace{.2cm}
\end{center}
\end{minipage}
\begin{minipage}{7.5cm}
\begin{center}
\vspace{-.1cm}
\begin{tikzpicture}
\tikzmath{\x1 = .6; \x2 = 1.6; \x3=5; \x4=2.5; \x5=1;  \x9=0.45;}
\draw [line width=\x9mm] (0,1.5*\x1+\x2) -- (\x4, 1.5*\x1+\x2);

\draw [->, line width=\x9mm] (0,0) -- (0,\x1);
\draw [->, line width=\x9mm] (0,\x1) -- (0,1.5*\x1+\x2);
\draw [->, line width=\x9mm] (\x3,0) -- (\x3,\x1);
\draw [->, line width=\x9mm] (\x3,\x1) -- (\x3,\x1+\x2);
\draw [->, line width=\x9mm] (\x3,\x1+\x2) -- (\x3,2*\x1+\x2);

\draw [line width=\x9mm, pattern=north east lines, pattern color=gray]  (\x3, \x1+\x2) -- (\x4+\x5,\x1+\x2) to[out=178,in=-30]  (\x4,1.5*\x1+\x2) to[out=30,in=182] (\x4+\x5,2*\x1+\x2) -- (\x3, 2*\x1+\x2) -- cycle;

\draw [line width=\x9mm, pattern=north east lines, pattern color=gray]  (0,0) -- (\x3,0) -- (\x3,\x1) -- (0,\x1) -- cycle;

\node[anchor =center] at (0,3*\x1+\x2) {};
\node[anchor = west] at (0.45*\x3,\x1+0.5*\x2) {$\CO$};
\node[anchor = south] at (0.2*\x3, \x1) {$C$};
\node[anchor = north] at (0.8*\x3, \x1+\x2) {$C$};
\node[anchor = north] at (0.2*\x3, 0) {$L_0$};
\node[anchor = south] at (0.8*\x3, 1.9*\x1+\x2) {$L_1$};
\node[anchor = south] at (0.2*\x3, 1.4*\x1+\x2) {$\RFF{1}$};
\node[anchor = west] at (0.45*\x3, 0.5*\x1) {$M$};
\node[anchor = center] at (0.8*\x3, 1.5*\x1+\x2) {$M$};

\end{tikzpicture}

\vspace{0cm}
\noindent \hypertarget{fig2}{FIGURE 2.}
\vspace{.2cm}
\end{center}
\end{minipage}
\end{center}

\noindent The second map $PSS$ is the relative version of the famous map of Piunikhin-Salamon-Schwarz \cite{PSS} and is constructed in \cite{BCP}. It is defined by counting configurations shown in Figure 3, where the arrows indicate Morse trajectories of $h$, the shaded regions, except the rightmost one, pseudoholomorphic disks in $\CO$ bounding $\RF$, and the rightmost shaded region a pseudoholomorphic once-punctured disk in $\CO$ bounding $\RF$ converging to an input \Ham chord at the puncture.

\begin{center}
\vspace{.3cm}
\begin{tikzpicture}
\tikzmath{\x1 = 2; \x2 = 1.5; \x3=1.2; \x4=1.2; \x5=1.8; \x6=1.2; \x7=1; \x8=0.75; \x9=1.2;}
\tikzmath{\y1=0.3;\y2=.45; \y3= \x1+2*\x2+\x3+\x4+\x5+\x6+\x8+\x9+\y1;}

\draw [line width=\y2mm] (0,0) -- (0.5*\x1,0) ;
\draw [->, line width=\y2mm] (\x1,0) -- (0.5*\x1,0) ;

\draw [line width=\y2mm] (\x1+\x2,0) -- (\x1+\x2+0.5*\x3,0) ;
\draw [->, line width=\y2mm] (\x1+\x2+\x3,0) -- (\x1+\x2+0.5*\x3,0) ;

\draw [line width=\y2mm] (\x1+2*\x2+\x3,0) --  (\x1+2*\x2+\x3+0.5*\x4,0) ;
\draw [->, line width=\y2mm] (\x1+2*\x2+\x3+\x4,0) -- (\x1+2*\x2+\x3+0.5*\x4,0)  ;

\node at  (\x1+2*\x2+\x3+\x4+0.25*\x5,0) [circle,fill,inner sep=1pt]{};
\node at  (\x1+2*\x2+\x3+\x4+0.5*\x5,0) [circle,fill,inner sep=1pt]{};
\node at  (\x1+2*\x2+\x3+\x4+0.75*\x5,0) [circle,fill,inner sep=1pt]{};

\draw [line width=\y2mm] (\x1+2*\x2+\x3+\x4+\x5,0) --  (\x1+2*\x2+\x3+\x4+\x5+0.5*\x6,0) ;
\draw [->, line width=\y2mm] (\x1+2*\x2+\x3+\x4+\x5+\x6,0) -- (\x1+2*\x2+\x3+\x4+\x5+0.5*\x6,0) ;

\draw [->, line width=\y2mm]  (\y3,-0.5*\x7) -- (\y3,0.5*\x7);
\draw [line width=0.2mm] (-0.1,0) ellipse (0.1 and 0.1);

\draw [line width=\y2mm, pattern=north east lines, pattern color=gray] (\x1+0.5*\x2,0) ellipse (0.5*\x2 and 0.5*\x2);
\draw [line width=\y2mm, pattern=north east lines, pattern color=gray] (\x1+1.5*\x2+\x3,0) ellipse (0.5*\x2 and 0.5*\x2);

\draw [line width=\y2mm, pattern=north east lines, pattern color=gray]  (\y3,0.5*\x7) -- (\y3-\y1,0.5*\x7)  to[out=178,in=-2] (\y3-\y1-\x9,\x8) arc (90:270:\x8) to[out=2,in=182]  (\y3-\y1,-0.5*\x7) -- (\y3,-0.5*\x7) -- cycle;
\end{tikzpicture}

\vspace{.3cm}
\noindent \hypertarget{fig3}{FIGURE 3.}
\end{center}
\noindent Notice that the target of $Y$ and the source of $PSS$ can be made chain-level isomorphic by choosing the \Ham for the pair $(\RFF{0},\RFF{0})$ to be $X\mapsto \langle X,a\rangle$. It follows that $PSS\circ Y$ is well-defined. This map will be useful when we prove Theorem \ref{main}.

\subsection{Perturbation data: Hamiltonians} \label{perturbdataH}
We introduce a quadratic \Ham
\[
\begin{array}{cccl}
H: &M&\ra& \RR\\
&[g:Y] &\mapsto & \frac{1}{2} |Y|^2
\end{array}.
 \]
Put $p_0:=eK\in G/K$ and choose a generic point\footnote{It is proved in \cite{BS} that (I) and (II) hold generically. As for (III), one applies Sard's theorem to the canonical map $\sq^{-1}(\RF)\ra G/K$.} $p_1\in G/K$ satisfying
\begin{condition}\label{nondeg} $~$
\begin{enumerate}[(I)]
\item every geodesic from $p_0$ to $p_1$ is non-degenerate;
\item $p_0$ and $p_1$ are contained in a unique maximal torus in $G/K$; and
\item the Lagrangians $\RFF{0}:=\RFF{p_0}$ and $\RFF{1}:=\RFF{p_1}$ intersect transversely.
\end{enumerate}
\end{condition}
\noindent Using the $K$-action on $G/K$, we may assume $p_1\in T_0$ so that $T_0$ is the unique maximal torus in $G/K$ containing $p_0$ and $p_1$, by (II). Write $p_1=\exp(a)K$ for some $a\in \t_0$.

Put $L_0:=L_{p_0}$ and $L_1:=L_{p_1}$. For $i=0,1$, denote by $\ul{L}_i$ either $(L_i,C)$ or $\RFF{i}$. We define the Floer cochain complexes $CW_b^*(L_0,L_1)$ and $Hom_{\AA}^*(\ul{L}_0,\ul{L}_1)$ using the (generalized) Hamiltonians $H_t^M:=H$ and
\[\ul{H}=(H_t^{M,-},H_t^{\CO},H_t^{M,+}):=(H,0,H)\]
respectively. Denote by $\X(L_0,L_1)$ the set of time-1 \Ham chords of $H_t^M$ from $L_0$ to $L_1$, and by $\X(\ul{L}_0,\ul{L}_1)$ the set of time-$(\delta_0,1,\delta_1)$ generalized \Ham chords of $\ul{H}$ for the cyclic set $(\ul{L}_0,(\ul{L}_1)^T)$ of Lagrangian correspondences, where $\delta_i=\frac{1}{2}$ or $0$, depending on whether $\ul{L}_i$ is equal to $(L_i,C)$ or $\RFF{i}$. Recall the Weyl group $W:=W_{X_0}$ and the lattice $Q^{\vee}\subset \t_0$ (Section \ref{Lie-basic} and \ref{Lie-sym}).

\begin{lemma}
\begin{enumerate}
\item[]
\item \[\X(L_0,L_1) = \{x_q|~q\in Q^{\vee}\}\]
where $x_q:[0,1]\ra M$ is defined by
\[ x_q(t):= [ \exp(t(q+a)):q+a] .\]
\item \[ \X(\ul{L}_0,\ul{L}_1) = \{\ul{x}_w|~w\in W\} \]
where $\ul{x}_w=(x_{w,-},x_{w,\CO},x_{w,+})$ with
\[ x_{w,-}:[0,\delta_0]\ra M,~\quad x_{w,\CO}:[0,1]\ra \CO,\quad x_{w,+}:[0,\delta_1]\ra M \]
such that $x_{w,\CO}\equiv wX_0$ and $x_{w,-}$ (resp. $x_{w,+}$) is the time-$\delta_0$ (resp. time-$\delta_1$) \Ham chord of $H$ which starts (resp. ends) at the unique intersection point $\sq^{-1}(wX_0)\cap L_0$ (resp. $\sq^{-1}(wX_0)\cap L_1$).
\end{enumerate}
Moreover, these chords are non-degenerate.
\end{lemma}
\begin{proof}
(1) follows from Condition \ref{nondeg}(II), Lemma \ref{flattorus} and the fact that every geodesic in $G/K$ is contained in a maximal torus. (2) follows from the facts that the \Ham flow of $H$ commutes with the symplectic quotient map $\sq:\C\ra\CO$ and
\[\RFF{0}\cap\RFF{1}=\{wX_0|~w\in W\},\]
which is an immediate consequence of Condition \ref{nondeg}(II) and Lemma \ref{Liewk2}(1). The non-degeneracy of these chords follows from Condition \ref{nondeg}(I) and (III).
\end{proof}

\noindent Let $\ul{x}_w\in\X(\ul{L}_0,\ul{L}_1)$. For any $q\in Q^{\vee}$, define $u_q\equiv wX_0$ and $\g_q$ to be any path in $\sq^{-1}(wX_0)$ such that the path $x_{w,-}\#\g_q\# x_{w,+}$ is continuous and its lift in the universal cover $T^*\t_0\ra T^*T_0$ is homotopic rel endpoints to $t\mapsto (t(q+a),wX_0)$. Then $(u_q,\g_q)$ is a \capping disk for $\ul{x}_w$.

\begin{lemma}\label{allmcd} $[u_q,\g_q],~q\in Q^{\vee}$ represent all homotopy classes of \capping disks for $\ul{x}_w$.
\end{lemma}
\begin{proof}
It follows from the homotopy lifting property of the fibration $\sq:\C\ra\CO$.
\end{proof}

\begin{corollary} Put $y_{w,q}:= (\ul{x}_w,[u_q,\g_q])$. We have
\begin{align*}
CW_b^*(L_0,L_1) & = \bigoplus_{q\in Q^{\vee}}\coeff \langle x_q\rangle\\
Hom_{\AA}^*(\ul{L}_0,\ul{L}_1) &= \bigoplus_{(w,q)\in W\times Q^{\vee}}\coeff\langle y_{w,q}\rangle.
\end{align*}
\end{corollary}

We also consider the cochain complex $QC^*(\RF)$ where the Morse function is taken to be $h:=\langle -,a \rangle$. It is straightforward to see that $h$ is indeed a Morse function with
\[\crit(h)=\t_0\cap\RF=\{wX_0|~w\in W\}.\]
Moreover, the set of homotopy classes of \capping disks for each of its critical points is determined in exactly the same way as in Lemma \ref{allmcd}. By an abuse of notation, we denote the generators $(wX_0,[u_q,\g_q])$ by $y_{w,q}$ so that
\[ QC^*(\RF) = \bigoplus_{(w,q)\in W\times Q^{\vee}}\coeff\langle y_{w,q}\rangle.\]

\begin{proposition} \label{nodiffprop} Suppose we are in one of the following situations:
\begin{enumerate}
\item The multiplicity $\ma$ is even for any $\a\in\Ra{X_0}$. This holds in the group case or when $(G,K)$\footnote{For simplicity, only irreducible symmetric pairs $(G,K)$ with $G$ simply connected are listed.} is one of
\[(SU(2n),Sp(n)),\quad (Spin(2n),Spin(2n-1)),\quad (E_6,E_4).\]
See \cite{Araki}.
\item The coefficient ring $\coeff$ is equal to $\ZZ_2$.
\end{enumerate}
Then the differentials of $CW_b^*(L_0,L_1), Hom_{\AA}^*(\ul{L}_0,\ul{L}_1)$ and $QC^*(\RF)$ vanish.
\end{proposition}
\begin{proof}
First notice that the cochain complexes $Hom_{\AA}^*(\ul{L}_0,\ul{L}_1)$ and $QC^*(\RF)$ are quasi-isomorphic and have the same set of generators. It thus suffices to prove the assertion for one of them. The assertion for the first situation follows from the index formulae \eqref{Index} for strips and triple-strips which define the differentials of $CW_b^*(L_0,L_1)$ and  $Hom_{\AA}^*(\ul{L}_0,\ul{L}_1)$ respectively. See Remark \ref{Indstrip}. For the second, recall that Bott-Samelson \cite{BS} constructed explicit bases for $HW_b^*(L_0,L_1)$ (which is isomorphic to $H_{-*}(\O(G/K);\ZZ_2)$) and $H^*(\RF;\ZZ_2)$. One can show that the dimension of each graded piece of these $\ZZ_2$-vector spaces is equal to that of $CW_b^*(L_0,L_1)$ and $QC^*(\RF)$ (without \capping disks) respectively. The key point is that the underlying combinatorics which describe their bases (homology level) and ours (chain level) are the same, namely the affine Weyl group for $\O(G/K)$ and the Weyl group for $\RF$. This shows that (a) the differential of $CW_b^*(L_0,L_1)$ vanishes, and (b) the count of Morse trajectories which contributes to the classical term of the differential of $QC^*(\RF)$ is zero. It remains to show that the quantum corrections for the latter differential is zero as well, and this follows from the standard argument, using the anti-symplectic involution $-D\involute$ on $\CO$. An extra argument is needed here: every pearly trajectory cancels its reflection as \capping disks. This follows from Lemma \ref{cancel} below which is an immediate consequence of Lemma \ref{allmcd}.
\end{proof}

\begin{lemma}  \label{cancel} The map
\[ (u,\g) \mapsto (-D\involute\circ u\circ r, \involute'\circ \g^{-1})\]
descends to the identity of the set of homotopy classes of \capping disks for any point of $\RF$, where $r$ is a suitable reflection on the domain and $\involute':[g:Y]\mapsto [\involute(g):Y]$.
\hfill $\square$
\end{lemma}

\begin{remark} Situation (1) can in fact be more general, as long as the ``degree argument'' works. Examples include all $G/K$, except $S^2$ and $\RR P^2$, which are of rank one, i.e. $\dim\t_0=1$.
\end{remark}

\subsection{Perturbation data: almost complex structures} \label{perturbdataJ}
We first introduce an almost complex structure $J_0$ on $M$. The almost complex structures for Floer theory will be small $C^0$-perturbations of $J_0$ satisfying some extra conditions. Recall (Section \ref{Lie-sym}) we have fixed a $G$-invariant metric on $G/K$ and denoted the associated Levi-Civita connection by $\nabla$. We have the splitting $TM=\Hor\oplus \Ver$ into horizontal and vertical subbundles.

\begin{definition}\label{defineJ} Define an almost complex structure $J_0$ on $M$ by the matrix
\begin{center}
\begin{tabular}{rl}
&\hspace{.2cm} $\Hor$ $~\Ver$ \\ \\ [-1em]
$\Hor$\hspace{-.3cm} & \multirow{2}{*}{
$ \begin{bmatrix}
0 & \id\\ -\id &0
\end{bmatrix}.$
 }\\
$\Ver$\hspace{-.3cm} &
\end{tabular}
\end{center}
\end{definition}

\begin{lemma}\label{J}$~$
\begin{enumerate}
\item $J_0$ is $\w_M$-compatible and of rescaled contact type in the sense of \cite{Ganatra}, i.e.
\[ \l_M\circ J_0 = rdr = dH\]
where $r:[g:Y]\mapsto |Y|$ is the length function.
\item The subbundles $\UU_0$ and $\UU_{\a}$ (Definition \ref{Ua}) are preserved by $J_0$.
\item If $T$ is a maximal torus in $G/K$, then $J_0$ preserves the tangent spaces of $T^*T$ and restricts to minus the standard one, i.e. $J_0$ pull-backs to $-J_{std}$ on the universal cover $T^*\RR^r\simeq \CC^r$ ($r=\dim T$) where the real part is identified with the zero section.

\end{enumerate}
\end{lemma}
\begin{proof}
(1) is standard. In fact, it holds for any metric on the base. (2) simply follows from the definitions of $\UU_0$ and $\UU_{\a}$. (3) is a special case of (2) since $\ol{\UU}=M$ and $\UU\cap T^*T$ is the disjoint union of a finite number of the leaves of the distribution $\UU_0$. It can also be proved by noticing that $T$ is a totally geodesic submanifold and the induced metric on $T$ is flat (Lemma \ref{flattorus}).
\end{proof}

\begin{corollary} \label{fcnpositive} Consider the (degenerate) 2-form $\w'':=\phi^*(\left.\w_M\right|_{\C})$ on $\UU$. The function
\[
\begin{array}{ccl}
T\UU &\ra& \RR\\
v&\mapsto & \w''(v,J_0 v)
\end{array}
 \]
is non-negative and vanishes precisely on the tangent spaces of the fibers of $\pj$.
\end{corollary}
\begin{proof}
It is clear that this function vanishes on the fibers. The rest follows from the decomposition \eqref{rootspglobaleq}, Proposition \ref{comparesymplform}(2) and Lemma \ref{J}.
\end{proof}

Choose a subset $A$ of $\UU$ satisfying the following conditions:
\begin{enumerate}
\item $A$ is closed in $M$;
\item the interior of $A$ contains $\C$ and every \Ham chord $x_q\in\X(L_0,L_1)$; and
\item there exist sequences $\{R_i\}$ and $\{\e_i\}$ of positive reals such that $R_i$ increases to $+\infty$ and
\[A\cap \{ |Y|\in (R_i-\e_i, R_i+\e_i)\}=\emptyset .\]
\end{enumerate}
To construct $A$, choose a suitable subset of $\{[e:Y]|~Y\in\wc_{X_0}\}$ and apply the $G$-action to it.

\begin{definition} \label{condJ} Define
\begin{enumerate}
\item $\J_A(M)$ to be the space of almost complex structures $J$ on $M$ satisfying
\begin{enumerate}[(i)]
\item $J\equiv J_0$ outside $A$;
\item $J$ preserves the distribution $\UU_0$;
\item $J|_{T^*T_0}\equiv J_0|_{T^*T_0}$; and
\item $J$ is pointwise close enough to $J_0$ such that
\begin{itemize}
\item $J$ is $\w_M$-tame; and
\item the function
\[
\begin{array}{ccl}
T\UU &\ra& \RR\\
v&\mapsto & \w''(v,J v)
\end{array}
 \]
is non-negative and vanishes precisely on $\UU_0$ ($\w''$ is defined in Corollary \ref{fcnpositive}).
\end{itemize}
\end{enumerate}
\item $\J(\CO)$ to be the space of $\w_{\CO}$-compatible almost complex structures on $\CO$.
\end{enumerate}
\end{definition}

\begin{remark} The closeness condition in (1)(iv) is justified as follows. First it is clear that for any point $[g:Y]\in M$, there is a neighbourhood of $J_0$ in $End(T_{[g:Y]}M)$ for which the $\w_M$-tameness condition holds. Now assume $[g:Y]\in\UU$. Let $V$ be the symplectic orthogonal of $\UU_0$ which is the direct sum of $\UU_{\a}$, $\a\in\Ra{X_0}$. Let $S\subset V$ be the unit sphere. Considering $\w''(v,Jv)$ as a continuous function on $S\times End(T_{[g:Y]}\UU)$, one sees that $J_0$ has a neighbourhood for which $\w''(v,Jv)>0$ for any $v\in S$, by Corollary \ref{fcnpositive}. If $J$ lies in this neighbourhood and satisfies (ii) above, then the second part of (iv) holds.
\end{remark}

\subsection{Pseudoholomorphic curves} \label{curve}
Having specified the Hamiltonians and almost complex structures in the previous two subsections, we proceed to describe three different types of pseudoholomorphic curves which are involved in the proof of our main theorem. In the proof, we will not consider $\Phi_C$ but the chain map (defined below)
\[\Phi'_C:CW_b^*(L_0,L_1)\ra Hom_{\AA}^*((L_0,C),(L_1,C)) .\]
\begin{itemize}
\item \textbf{The differential of $\mathbf{CW_b^*(L_0,L_1)}$.} Choose Floer data $(H_t^M,J_t^M)$ with
\[ H_t^M \equiv H,\quad\quad  J_t^M\in \J_A(M).\]
The differential counts index-one solutions, modulo the $\RR$-action, $u:\RR\times[0,1]\ra M$ satisfying
\begin{condition}\label{S}
\end{condition}
\begin{center}
\begin{tabular}{lc}
(boundary condition)& $u(\cdot,0)\in L_0,~u(\cdot,1)\in L_1$\\ \\ [-.5em]
(asymptotic condition)& \qquad
\begin{minipage}{7cm}
$u$ converges to an element of $\X(L_0,L_1)$ as $s\to+\infty$ and $s\to -\infty$
\end{minipage}\\ \\ [-.5em]
(differential equation)& $ (du-X_{H_t^M}\otimes dt)^{0,1}_{J_t^M}=0$\\ \\ [-.5em]
(finiteness of energy)& $ \int |du-X_{H_t^M}\otimes dt|^2 < +\infty$.
\end{tabular}
\end{center}

\item \textbf{The differential of $\mathbf{Hom_{\AA}^*((L_0,C),(L_1,C))}$.} Let $\ul{S}=(S_-,S_{\CO},S_+)$ with $S_{\pm}:=\RR\times [0,\frac{1}{2}]$ and $S_{\CO}:=\RR\times[0,1]$. The Floer data $(\ul{H},\ul{J})$ consist of $(H_t^{M,\pm},J_t^{M,\pm})$ and $(H_t^{\CO},J_t^{\CO})$ on $S_{\pm}$ and $S_{\CO}$ respectively which we choose to be
\begin{center}
\begin{minipage}{3cm}
\begin{align*}
H_t^{M,-} &\equiv H,\\
H_t^{\CO}& \equiv 0,\\
H_t^{M,+} &\equiv H,
\end{align*}
\end{minipage}
\begin{minipage}{3cm}
\begin{align*}
J_t^{M,-} &\in \J_A(M),\\
 J_t^{\CO} &\in \J(\CO),\\
J_t^{M,+} &\in \J_A(M).
\end{align*}
\end{minipage}
\end{center}
The differential counts index-one solutions, modulo the $\RR$-action, $\ul{u}=(u_-,u_{\CO},u_+)$ with $u_{\pm}:S_{\pm}\ra M$ and $u_{\CO}:S_{\CO}\ra \CO$ satisfying
\begin{condition}\label{TS}
\end{condition}
\begin{center}
\begin{tabular}{lc}
(boundary condition)&$u_-(\cdot,0)\in L_0,~u_+(\cdot,1/2)\in L_1$ \\ \\ [-.5em]
(seam condition)&$(u_-(\cdot,1/2),u_{\CO}(\cdot,0))\in C,~ (u_{\CO}(\cdot,1),u_+(\cdot,0))\in C^T$ \\ \\ [-.5em]
(asymptotic condition)& \qquad
\begin{minipage}{7cm}
$\ul{u}$ converges to an element of $\X(\ul{L}_0,\ul{L}_1)$ as $s\to+\infty$ and $s\to -\infty$
\end{minipage}\\ \\ [-.5em]
(differential equation)& $ (d\ul{u}-X_{\ul{H}}\otimes dt)^{0,1}_{\ul{J}}=0$\\ \\ [-.5em]
(finiteness of energy)& $ \int |d\ul{u}-X_{\ul{H}}\otimes dt|^2 < +\infty$.
\end{tabular}
\end{center}

\item \textbf{The chain map $\mathbf{\Phi'_C:CW_b^*(L_0,L_1)\ra Hom_{\AA}^*((L_0,C),(L_1,C))}$. }
Recall \cite{EL, MWW} it counts pseudoholomorphic quilts whose domain is a pair $\ul{\S}=(\S_M,\S_{\CO})$ illustrated in Figure 4 below.
\begin{center}
\vspace{.1cm}
\begin{tikzpicture}
\tikzmath{\x1 = 0.6; \x2 = 1.6; \x3=2; \x4=1; \x5=\x1+(\x2-\x4)/2; \x6=1.5; \x7=3.5; \x8=0.2; \x9=0.45; \y1=0.45;}

\draw [->, line width=\x9mm] (0,0) -- (0,\x1) node [midway, left] {$x_{w,-}$};
\draw [->, line width=\x9mm] (0,\x1) -- (0,\x1+\x2) node [midway, left] {$x_{w,\CO}$};
\draw [->, line width=\x9mm] (0,\x1+\x2) -- (0,\x1+\x2+\x1) node [midway, left] {$x_{w,+}$};

\draw [->, line width=\x9mm] (\x3+\x7+\x6,\x5) -- (\x3+\x7+\x6,\x5+\x4) node [midway, right] {$x_q$};

\draw [line width=\x9mm, pattern=north east lines, pattern color=gray]  (0,0) -- (\x3,0)  to[out=2,in=182] (\x3+\x7,\x5) -- (\x3+\x7+\x6,\x5) -- (\x3+\x7+\x6,\x5+\x4) -- (\x3+\x7,\x5+\x4) to[out=178,in=-2]  (\x3,\x1+\x2+\x1) -- (0,\x1+\x2+\x1) -- (0,\x1+\x2) -- (\x3-\x8,\x1+\x2) arc (90:-90:0.5*\x2) -- (\x3-\x8,\x1) -- (0,\x1) -- cycle;

\node[anchor = south west] at (\x5+1.6*\x7/2,\x1+\x2+0.2*\x1/2) {$\pt_{L_1}\S_M$};
\node[anchor = north west] at (\x5+1.6*\x7/2,\x1-0.3*\x1/2)  {$\pt_{L_0}\S_M$};
\node[anchor = west] at (0.4*\x3, \x1+\x2/2) {$\S_{\CO}$};
\node[anchor = east] at (\x3+0.95*\x7, \x1+\x2/2) {$\S_M$};
\node[anchor = east] at (\x3+0.55*\x7, \x1+\x2/2) {$\pt_C\S_M$};
\end{tikzpicture}

\vspace{.2cm}
\noindent \hypertarget{fig4}{FIGURE 4.}
\end{center}

\noindent There is a positive \sle in the right (as input, from $CW_b^*(L_0,L_1)$) and a negative triple \sle in the left (as output, to $Hom_{\AA}^*((L_0,C),(L_1,C))$). The boundary $\pt\S_M$ of $\S_M$ has three components which are required to be real analytic. We label the bottom one by $\pt_{L_0}\S_M$, the top one by $\pt_{L_1}\S_M$ and the remaining one, which is also the boundary $\pt\S_{\CO}$ of $\CO$ (with reversed orientation), by $\pt_C\S_M$. This labelling will be compatible with the boundary and seam conditions we are going to impose on the quilts.

The Floer data consist of $(H_M,J_M), (H_{\CO},J_{\CO})$ and a closed 1-form $df$ which we choose to be exact. We require
\begin{enumerate}
\item $H_M\equiv H$ and $H_{\CO}\equiv 0$;
\item $J_M\in \mathcal{C}^{\infty}(\S_M,\J_A(M))$ and restricts to $J_t^M$ (resp. $J_t^{M,\pm}$) near the positive (resp. negative) strip-like end;
\item $J_{\CO}\in \mathcal{C}^{\infty}(\S_{\CO},\J(\CO))$ and restricts to $J_t^{\CO}$ near the negative strip-like end; and
\item the function $f:\S_M\ra \RR$ satisfies
\begin{enumerate}
\item $f|_{\pt_{L_0}\S_M}\equiv 0$, $f|_{\pt_C\S_M}\equiv \frac{1}{2}$ and $f|_{\pt_{L_1}\S_M}\equiv 1$;
\item $f$ is linear with $t$ near all strip-like ends.
\end{enumerate}
\end{enumerate}

\noindent The chain map $\Phi'_C$ counts index-zero solutions $\ul{u}=(u_M,u_{\CO})$ with $u_M:\S_M\ra M$ and $u_{\CO}:\S_{\CO}\ra \CO$ satisfying
\begin{condition}\label{Q}
\end{condition}
\begin{center}
\begin{tabular}{lc}
(boundary condition)&$u_M(\pt_{L_0}\S_M)\subset L_0,~u_M(\pt_{L_1}\S_M)\subset L_1$ \\ \\ [-.5em]
(seam condition)& \quad
\begin{minipage}{7cm}
$(u_M(z),u_{\CO}(z))\in C$ for every $z\in \pt_C\S_M$ (also regarded as an element of $\pt\S_{\CO}$)
\end{minipage}\\ \\ [-.5em]
(asymptotic condition)& \quad
\begin{minipage}{7.2cm}
$\ul{u}$ converges to an element of $\X(L_0,L_1)$ at the positive \sle and an element of $\X(\ul{L}_0,\ul{L}_1)$ at the negative triple strip-like ends
\end{minipage}\\ \\ [-.5em]
(differential equation)&
\begin{minipage}{7cm}
\[\left\{
\begin{array}{rc}
(du_M-X_H\otimes df)^{0,1}_{J_M} ~=& 0\\ \\ [-1em]
(du_{\CO})^{0,1}_{J_{\CO}}~=&0
\end{array}
\right.\]
\end{minipage}\\ \\ [-.5em]
(finiteness of energy)& $ \int |du_M-X_H\otimes df|^2 + \int |du_{\CO}|^2 < +\infty$.
\end{tabular}
\end{center}
\end{itemize}

\section{Proof of the main theorem} \label{pf}
In this section, by a quilt we always mean a quilt satisfying Conditions \ref{Q}.
\subsection{Analysis of the morphisms} \label{pfA}
Let $\ul{u}=(u_M,u_{\CO})$ be a quilt. Let $\tilde{u}_{\CO}$ be a lift of $u_{\CO}$ in $\C$ (recall $\C$ is a torus bundle over $\CO$) such that $v:=\tilde{u}_{\CO}\# u_M$ is continuous as a map $\S_M\cup \S_{\CO}\ra M=T^*(G/K)$. Denote by $x_{\ul{u},in}$ (resp. $x_{\ul{u},out}$) the limiting paths of $v$ near the positive (resp. negative) strip-like end. Then $x_{\ul{u},in}\in\X(L_0,L_1)$ and $x_{\ul{u},out}$ is the concatenation of two time-$\frac{1}{2}$ \Ham chords of $H$ and a continuous path connecting them. As $x_{\ul{u},in}$ and $x_{\ul{u},out}$ lie in $T^*T_0$, we can project them to the zero section and obtain two paths from $p_0=eK$ to $p_1=\exp(a)K$. Then there are unique $q_{\ul{u},in}, q_{\ul{u},out}\in Q^{\vee}$ such that these paths are homotopic, with fixed endpoints, to the paths $t\mapsto \exp(t( q_{\ul{u},in} +a ))\cdot p_0$ and $t\mapsto \exp(t( q_{\ul{u},out} +a))\cdot p_0$ respectively. Denote by $X_{\ul{u},in}, X_{\ul{u},out}\in W\cdot X_0$ the unique elements such that $\phi(x_{\ul{u},in})$ and $x_{\ul{u},out}$ lie in the fibers of $\sq$ at $X_{\ul{u},in}$ and $X_{\ul{u},out}$ respectively. From now on, we drop the subscript $\ul{u}$ in every notation we have just introduced, since it will be unlikely to cause any confusion.

\begin{definition}\label{goodbadugly} Fix a small open disk $V$ in $\S_M\cup\S_{\CO}$ which intersects the seam $\pt_C\S_M=\pt\S_{\CO}$. Let $\ul{u}=(u_M,u_{\CO})$ be a quilt.
\begin{enumerate}
\item We say that $\ul{u}$ is {\it good} if $V$ contains two points $z_1, z_2$ with $z_1\in \S_M,~z_2\in\S_{\CO}$ and away from the seam such that $du_{\CO}(z_2)\neq 0$, $u_M(z_1)\in \mr{A}$ and $(du_M-X_H\otimes df)(z_1)$ is not tangent to the fiber of $\pj$ containing $u_M(z_1)$.

\item We say that $\ul{u}$ is {\it bad} if $u_{\CO}$ is constant and the image of $u_M$ is contained in a fiber of $\pj$.

\item We say that $\ul{u}$ is {\it ugly} if $u_{\CO}$ is constant and the image of $u_M$ lies in $T^*T_0$, but it is not bad.
\end{enumerate}
\end{definition}

\noindent The following propositions will be proved in Appendix \ref{trans}.
\begin{proposition}\label{A} Every quilt must be of exactly one of the three types: good, bad and ugly.
\end{proposition}

\begin{proposition}\label{B} For generic $\ul{J}=(J_M,J_{\CO})$, every good quilt $\ul{u}=(u_M,u_{\CO})$ is regular and $u_M$ intersects $M\setminus\UU$ at finitely many points in the interior of $\S_M$.
\end{proposition}

\begin{proposition}\label{B+C/2} For any $w\in W$ and $q\in Q^{\vee}$, there exists a unique quilt $\ul{u}=(u_M,u_{\CO})$ such that
\begin{enumerate}
\item $u_{\CO}$ is constant and $\im(u_M)\subset T^*T_0$; and
\item its input is $x_q\in\X(L_0,L_1)$ and output is $\ul{x}_w\in\X(\ul{L}_0,\ul{L}_1)$.
\end{enumerate}
It is independent of $\ul{J}$; it is bad if $q+a\in w\wc_{X_0}$; otherwise it is ugly.
\end{proposition}

\begin{proposition}\label{C} All bad quilts have index zero and are regular.
\end{proposition}

\begin{proposition}\label{D} All ugly quilts have positive indices . It follows that they do not contribute to $\Phi'_C$. Moreover, the standard compactness argument for proving the finiteness of the moduli is applicable.
\end{proposition}

\noindent From now on, assume $\ul{J}=(J_M,J_{\CO})$ has been chosen generically.

\begin{lemma}\label{end1} If $\ul{u}$ is good, then
\begin{equation}\label{good>}
\langle q_{in} +a,X_{in}\rangle > \langle q_{out} +a,X_{out}\rangle.
\end{equation}
\end{lemma}

\begin{proof}
It is easy to see that the left-hand and right-hand sides of \eqref{good>} are equal to $\int (\phi\circ x_{in})^*\l_M$ and $\int x_{out}^*\l_M$ respectively. By Proposition \ref{B}, we can write $u_M^{-1}(M\setminus\UU)=\{z_1,\ldots,z_k\}\subset \mr{\S}_M$. For each $i=1,\cdots,k$, let $D_i$ be a small disk in $\mr{\S}_M$ centered at $z_i$, and put $\S_M':=\S_M\setminus\bigcup_{i=1}^k \mr{D}_i$. Then $u_M(\S_M')\subset \UU$. Since $\CO$ is the symplectic quotient of the coisotropic submanifold $\C$ and $\phi|_{\C}=\id$, we have
\begin{equation}\label{pf1}
\int_{\S'_M\cup\S_{\CO}} (\phi\circ v)^*\w_M = \int_{\S'_M}(\phi\circ u_M)^*\w_M + \int_{\S_{\CO}}u_{\CO}^*\w_{\CO}.
\end{equation}
By Stokes' theorem, the left-hand side of \eqref{pf1} is equal to
\[   \langle q_{in} +a,X_{in}\rangle - \langle q_{out} +a,X_{out}\rangle - \sum_{i=1}^k\int_{\pt D_i}(\phi\circ u_M|_{\pt D_i})^*\l_M. \]
We claim that each integral $\int(\phi\circ u_M|_{\pt D_i})^*\l_M$ converges to 0 as $\diam(D_i)\to 0$. To prove it, write $\l_M=\mb{p}d\mb{q}$, the standard local expression for the Liouville form $\l_M$. While $\phi\circ u_M$ may not be defined on $D_i$, its projection $\mb{q}\circ \phi\circ u_M$ into the base $G/K$ is. Since $\mb{p}\circ\phi\circ u_M$ is bounded on $D_i\setminus\{z_i\}$, the claim is proved.

Since $\ul{u}$ is good, we have $\int_{\S_{\CO}}u_{\CO}^*\w_{\CO}>0$. Therefore, it suffices to show that $\int_{\S'_M}(\phi\circ u_M)^*\w_M \geqslant 0$. Using the perturbed Cauchy-Riemann equation $(du_M-X_H\otimes df)^{0,1}_{J_M}=0$ and the fact that $\phi_*X_H$ lies in the kernel of $\w_M|_{\C}$, we have
\[ (\phi\circ u_M)^*\w_M = \phi^*\w_M (\pt_s u_M - \pt_s f\cdot X_H, J_M(\pt_s u_M - \pt_s f\cdot X_H) ) ds\wedge dt. \]
which is non-negative by the condition imposed on $J_M$. (See Definition \ref{condJ}.)
\end{proof}

\begin{lemma}\label{end2}
Suppose $\ul{u}$ satisfies $\ind(\ul{u})=0$. We have
\begin{equation}\label{good>3}
\sum_{\a\in R_{X_{in}}^+} \ma \{ 2\a(a)\} \geqslant \sum_{\a\in R_{X_{out}}^+} \ma \{ 2\a(a)\}
\end{equation}
and the equality holds if and only if $\ul{u}$ is bad. ($\{x\}:=x-\lfloor x\rfloor $ is the fractional part of $x$.)
\end{lemma}

\begin{proof}
By Proposition \ref{D}, $\ul{u}$ cannot be ugly. If $\ul{u}$ is bad, then $X_{in}=X_{out}$ so that \eqref{good>3} becomes the equality. If $\ul{u}$ is not bad, then it must be good by Proposition \ref{A}. By Lemma \ref{end1}, we have
\begin{equation}\label{good>2}
\langle q_{in} +a,X_{in}\rangle > \langle q_{out} +a,X_{out}\rangle.
\end{equation}

\noindent Recall the index formula \eqref{Index} which gives
\begin{equation}\label{index4.8}
\sum_{\a\in R_{X_{in}}^+}\ma \lfloor 2\a( q_{in} +a)\rfloor = \sum_{\a\in R_{X_{out}}^+} \ma \lfloor 2\a( q_{out}+ a)\rfloor
\end{equation}
and Assumption \ref{monotone} which is the essential condition for $\CO$ to be monotone:
\begin{equation}\label{endmonotone}
\langle - , X_0\rangle|_{\t_0} = 2\monoconst \sum_{\a\in\Ra{X_0}} \ma \a(-)
\end{equation}
where $\monoconst>0$ is a constant. Clearly \eqref{endmonotone} holds if $X_0$ is replaced by $wX_0$ for any $w\in W$. Hence
\[ \sum_{\a\in R_{X_{in}}^+} \ma \a(q_{in}+a) >  \sum_{\a\in R_{X_{out}}^+} \ma \a(q_{out}+a) .\]
It follows that
\begin{align*}
\sum_{\a\in R_{X_{in}}^+} \ma  \{2 \a(a)\} & = \sum_{\a\in R_{X_{in}}^+} \ma  \left( 2\a(q_{in}+a) - \lfloor 2\a(q_{in}+a) \rfloor\right) \\
& > \sum_{\a\in R_{X_{out}}^+} \ma \left( 2\a(q_{out}+a) - \lfloor 2\a(q_{out}+a) \rfloor\right) \\
& = \sum_{\a\in R_{X_{out}}^+} \ma  \{ 2\a(a)\} .
\end{align*}
\end{proof}

Define $\ell':W\ra\RR$ by
\[ \ell'(w) := \sum_{\a\in\Ra{wX_0}} \ma \{2 \a(a)\}.\]

\begin{lemma}\label{fil1} For any $q\in Q^{\vee}$,
\[ \Phi_C'(x_q)\in \pm y_{w_q,q} + \bigoplus_{ {\substack{(w,q')\in W\times Q^{\vee} \\ \ell'(w)<\ell'(w_q)}}  }\coeff\langle y_{w,q'}\rangle\]
where $w_q\in W$ is the unique element such that $q+a\in w_q\wc_{X_0}$.
\end{lemma}
\begin{proof}
It is not hard to see that the unique bad quilt from Proposition \ref{B+C/2} with input $x_q$ contributes to the leading term $\pm y_{w_q,q}$. The rest follows from Lemma \ref{end2}.
\end{proof}

\begin{lemma}\label{filY}
The quasi-isomorphism
\[Y:Hom_{\AA}^*( (L_0,C),(L_1,C))\ra Hom_{\AA}^*(\RFF{0},\RFF{1})\]
defined in Section \ref{capdisk} satisfies
\begin{equation}\label{filYeq}
Y(y_{w,q})\in \pm y_{w,q} + \bigoplus_{ {\substack{(w',q')\in W\times Q^{\vee} \\ \ell'(w')<\ell'(w)}}  }\coeff\langle y_{w',q'}\rangle
\end{equation}
for any $w\in W$ and $q\in Q^{\vee}$.
\end{lemma}
\begin{proof}
Recall $Y=Y_1^{-1}\circ Y_2$ where $Y_1$ and $Y_2$ are defined in Figure \hyperlink{fig1}{1} and Figure \hyperlink{fig2}{2} respectively. If $Y_1$ is shown to have the form \eqref{filYeq}, then it is bijective, and its inverse also has the same form. For any solution described in these figures, we lift the portion which lands in $\CO$ to $\C$ and glue it with the portion which lands in $M$. Then the result follows from similar arguments for the proof of Lemma \ref{end1} and Lemma \ref{end2}. Notice that we do not require the almost complex structures on $M$ to lie in $\J_A(M)$, and that constant solutions are regular, as proved in \cite{LL}.
\end{proof}

\begin{lemma}\label{filPSS} Assume $|a|$ small and $a\in \mr{\wc}_{X_0}$. Then the quasi-isomorphism
\[PSS: Hom_{\AA}^*(\RFF{0},\RFF{0})\ra QC^*(\RF)\]
defined in Section \ref{capdisk} satisfies
\begin{equation}
PSS(y_{w,q})\in \pm y_{w,q} + \bigoplus_{ {\substack{(w',q')\in W\times Q^{\vee} \\ \ell'(w')<\ell'(w)}}  }\coeff\langle y_{w',q'}\rangle
\end{equation}
for any $w\in W$ and $q\in Q^{\vee}$.
\end{lemma}
\begin{proof} As in the proof of Lemma \ref{end2}, it suffices to obtain \eqref{good>2} for non-constant solutions and \eqref{index4.8}. Let us deal with the latter first. The index condition for solutions described in Figure \hyperlink{fig3}{3} is given by
\begin{equation}\nonumber
\sum_{\a\in\Ra{X_{in}} }\ma \lfloor 2\a(q_{in}+a)\rfloor -  \sum_{\a\in\Ra{X_{out}} }2 \ma\a(q_{out}) + \ell(w_{out})=0
\end{equation}
where $X_{out}=w_{out}X_0$ and $\ell:W\ra \ZZ_{\geqslant 0}$ is the standard length function on the Weyl group $W$. It thus suffices to show
\[\ell(w_{out}) = -  \sum_{\a\in\Ra{X_{out}} } \ma \lfloor 2\a(a)\rfloor, \]
and this follows from
\[ \ell(w_{out})=\sum_{ \substack{ \a\in\Ra{X_{out}}\\ \a(a)<0 }}\ma \]
and the smallness assumption on $|a|$ which implies
\[  \lfloor 2\a(a)\rfloor = \left\{
\begin{array}{rl}
0, & \a(a)>0\\
-1, & \a(a)<0
\end{array}
\right. .\]

Now, we prove \eqref{good>2}. For each solution, denote the rightmost punctured disk by $u_0$ and the other disks by $u_1,\ldots,u_k$. Recall we have taken $h:=\langle -,a\rangle$ for the \Ham defining $Hom_{\AA}^*(\RFF{0},\RFF{0})$ and its restriction to $\RF$ for the Morse function for defining $QC^*(\RF)$. According to \cite{BCP}, these disks solve
\[ (du_0-\b(s)X_h\otimes dt)^{0,1}=0,\quad (du_i)^{0,1}=0,~i=1,\ldots, k  \]
where $\b:\RR\ra [0,1]$ is a fixed cut-off function which increases from 0 near $-\infty$ to 1 near $+\infty$. There are two cases.
\begin{enumerate}
\item[(H)] $k\geqslant 1$ or $\int u_0^*\w_{\CO}>0$;
\item[(L)] $k=0$ and $\int u_0^*\w_{\CO}\leqslant 0$.
\end{enumerate}

In case (H), the sum $\sum_{i=0}^k\int u_i^*\w_{\CO}$ is greater than a positive constant which is independent of $a$, by the monotonicity of $\RF$. It is not hard to show that this sum is equal to $\langle q_{in}, X_{in}\rangle -\langle q_{out}, X_{out}\rangle $, and hence the inequality \eqref{good>2} holds if $|a|$ is small enough.

In case (L), the geometric energy $E(u_0):=\int |du_0-\b(s)X_h\otimes dt|^2$ of $u_0$ is small if $|a|$ is small. We apply an argument of Oh \cite{Oh} to show that any such solution must be contained in a fixed Weinstein neighbourhood of $\RF$. This allows us to define an invariant by counting these solutions. By choosing particular perturbation data as in \cite{Floer}, we see that this invariant contributes only to the leading term of \eqref{filYeq}.
\end{proof}

\subsection{Central invertible elements of the target} \label{pfB}
Next we look at generators $y_{w,q}$ of $QC^*(\RF)$.
\begin{lemma}\label{prodform} Assume
\begin{enumerate}[(i)]
\item $G/K$ is $Spin$ if $\coeff\ne\ZZ_2$; and
\item $a\in \mr{\wc}_{X_0}$.
\end{enumerate}
Then the following hold:
\begin{enumerate}[(1)]
\item For any $q\in Q^{\vee}$, $y_{e,q}$ is closed with respect to the Floer differential of $QC^*(\RF)$.
\item Each of the chain maps
\[ -\star y_{e,q},~y_{e,q}\star -:QC^*(\RF)\ra QC^*(\RF)\]
sends $y_{w,q'}$ to $\pm y_{w,q'+w(q)}$ for any $w\in W$ and $q'\in Q^{\vee}$. In particular, $y_{e,q}$ descends to an invertible element of $QH^*(\RF)$.
\item The orientations on the determinant lines (see Appendix \ref{sign}) associated to the generators $y_{w,q}$ can be chosen consistently such that the sign in (2) is always $+1$. In particular, each $y_{e,q}$ descends to a central element of $QH^*(\RF)$.
\end{enumerate}
\end{lemma}
\begin{proof} Let us put aside the sign issue first. Observe that, under assumption (ii), $X_0$ is the unique critical point of the Morse function $\langle -, a\rangle$ whose descending submanifold has the ambient dimension. Biran-Cornea \cite{BCP} proved that, in the absence of \capping disks, this point is the chain-level unit. In the presence of \capping disks, consider the fiber bundle $\sq:\C|_{\RF}\ra \RF$. Recall that it has torus fibers and the fiber over each $wX_0$ is equal to a constant section of $T^*T_0\subset T^*(G/K)$ over $T_0$. By definition, $y_{w,q}$ corresponds to the homotopy class represented by a loop in $\sq^{-1}(wX_0)$ whose projection in $T_0$ is homotopic to the loop $t\mapsto\exp(tq)$.

Therefore, part (1) over $\ZZ_2$ and part (2) in general follow, if we can show that for any path $\l_t$ in $\RF$ joining $X_0$ and $wX_0$, the parallel transport along $\l_t$ brings a loop $\g_q$ over $X_0$ homotopic to $\exp(tq)$ to a loop over $wX_0$ which is homotopic to $\exp(tw(q))$. Write $\l_t=k_t\cdot X_0$ for some path $k_t$ in $K$ starting at the identity. Then $k_t\cdot \g_q$ is a cylinder in $\C|_{\RF}$ over $\l_t$. It suffices to show that $k_1\cdot \g_q$ is homotopic to $\exp(tw(q))$, and this follows from Lemma \ref{Liewk2}(1).

Suppose now that signs are taken into account. Denote by $\VTT$ and $\VT$ the vertical tangent bundles of the fiber bundles $\pi:\C\ra G/K$ and $\sq:\C\ra \CO$ respectively. We have
\begin{equation}\label{pfB1}
\pi^*(T(G/K))\simeq \VTT\oplus\VT.
\end{equation}
Notice that $\VTT$ is a Lagrangian subbundle of $\sq^*T\CO$ so that every loop $\g$ lying in a fiber $\sq^{-1}(X)$ is naturally endowed with a Cauchy-Riemann operator defined by $(T_X\CO,\VTT|_{\g})$. We claim that for any loop $\g$ in $\RF$, $\VT|_{\sq^{-1}(\g)}$ is trivial. Indeed, this follows from the fact that every torus is parallelizable and the above argument which uses Lemma \ref{Liewk2}(1).

It follows that, by assumption (i) and \eqref{pfB1}, $\VTT|_{\sq^{-1}(wX_0)}$ is $Spin$ for any $w\in W$ and $Spin$ structures on these bundles can be chosen consistently such that they are preserved under parallel transport along any path in $\RF$. By a standard result (e.g. \cite{FOOO, WWorient}), there exists a natural bijective correspondence between the set of orientations on the determinant lines associated to loops in the fibers of $\sq$ and the set of $Spin$ structures of $\VTT$ over these loops in which gluing of Cauchy-Riemann operators corresponds to concatenation of $Spin$ structures. This proves parts (1) and (3), and hence the lemma.
\end{proof}

\begin{lemma} \label{nosign} The orientations on the determinant lines associated to the generators can be chosen such that the signs in Lemma \ref{fil1}, Lemma \ref{filY} and Lemma \ref{filPSS} are equal to $+1$.
\end{lemma}
\begin{proof}
Observe that for each of these lemmas, the function assigning to every generator of the source of the map involved the corresponding leading term in the target is injective. (For example, the function for Lemma \ref{fil1} sends $x_q$ to $y_{w_q,q}$.) Therefore, the orientations for the generators of $QC^*(\RF)$ which are fixed in Lemma \ref{prodform} induce, in a unique way, the orientations for the generators of other cochain complexes such that the signs of the leading terms are equal to $+1$.
\end{proof}
\subsection{The final step} \label{pfC}
We now prove Theorem \ref{main} and Theorem \ref{maincor}. By Proposition \ref{nodiffprop}, the results in Lemma \ref{fil1}, Lemma \ref{filY}, Lemma \ref{filPSS}, Lemma \ref{prodform} and Lemma \ref{nosign} hold on cohomology level. Notice that $HW_b^*(L_0,L_1)$ is a right $HW_b^*(L_0,L_0)$-module.  Define $\rmor':HW_b^*(L_0,L_1)\ra QH^*(\RF)$ to be $H^*(PSS\circ Y\circ \Phi_C')$\footnote{$PSS\circ Y$ is well-defined as the Hamiltonian for the pair $(\RF_0,\RF_0)$ is chosen to be $X\mapsto \langle X,a\rangle$.} and $\rmor:HW_b^*(L_0,L_0)\ra QH^*(\RF)$ in a similar way. Then $\rmor$ is a $\coeff$-algebra homomorphism and $\rmor'$ is a module homomorphism with respect to $\rmor$ where $QH^*(\RF)$ acts on itself by right multiplication. We claim that the element $x_0~(0\in Q^{\vee})$ is a free generator of $HW_b^*(L_0,L_1)$ as $HW_b^*(L_0,L_0)$-module. Indeed, by \cite{Ab_JSG}, there are isomorphisms
\begin{align*}
HW_b^*(L_0,L_0) & \ra H_{-*}(\O_{p_0\to p_0}(G/K))\\
HW_b^*(L_0,L_1) & \ra H_{-*}(\O_{p_0\to p_1}(G/K))
\end{align*}
where the first map is a $\coeff$-algebra isomorphism and the second map is a module isomorphism with respect to the first map. Our claim becomes obvious if we transform our problem to one involving the targets of these maps. On the other hand, observe that $\ell':W\ra \RR$ attains its minimum precisely at $e\in W$. By Lemma \ref{fil1}, Lemma \ref{filY}, Lemma \ref{filPSS} and Lemma \ref{nosign}, $\rmor'(x_0)=y_{e,0}$ which is clearly a free generator of $QH^*(\RF)$ as $QH^*(\RF)$-module. Hence, Theorem \ref{main} and Theorem \ref{maincor} will follow from Proposition \ref{final} below which deals with the corresponding problems for the module homomorphism $\rmor'$ (instead of $\rmor$). Before stating it, we put $R_1:=HW_b^*(L_0,L_0)$, $R_2:=QH^*(\RF)$, $N_1:=HW_b^*(L_0,L_1)$ and $N_2:=QH^*(\RF)$. Define
\[ S_1:=\{x_q|~q\in Q^{\vee},  w_q=e\}\quad\text{ and }\quad S_2:=\{y_{e,q}|~q\in Q^{\vee},  w_q=e\} \]
considered as subsets of $R_1$ and $R_2$ via the isomorphisms $R_1\simeq N_1$ and $R_2\simeq N_2$ induced by the multiplications by $x_0$ and by $\rmor'(x_0)$ respectively. Notice that $S_2$ consists of central and invertible elements of $R_2$, by Lemma \ref{prodform}.

\begin{proposition}\label{final}$~$
\begin{enumerate}
\item $S_1$ is a multiplicative subset consisting of central elements and $\rmor(S_1)=S_2$.
\item $S_1^{-1}\rmor': S_1^{-1}N_1\ra N_2$ is a module isomorphism.
\item The sub-$\coeff$-algebra $\coeff[S_1]$ of $R_1$ generated by $S_1$ is Noetherian, and $N_1$ is a finitely generated $\coeff[S_1]$-module.
\end{enumerate}
\end{proposition}

\begin{proof} We endow $N_1$ and $N_2$ with filtrations $\{F^{\mu}N_1\}_{\mu\in\RR}$ and $\{F^{\mu}N_2\}_{\mu\in\RR}$ respectively, where
\[F^{\mu}N_1:=\bigoplus_{\substack{q\in Q^{\vee}\\ \ell'(w_q)\leqslant \mu}} \coeff\langle x_q\rangle \quad\text{ and }\quad  F^{\mu}N_2:=\bigoplus_{\substack{(w,q)\in W\times Q^{\vee}\\ \ell'(w)\leqslant \mu}} \coeff\langle y_{w,q}\rangle. \]
Notice that each of $F^{\mu}N_1$ and $F^{\mu}N_2$ jumps for only finitely many $\mu$. By Lemma \ref{fil1}, Lemma \ref{filY} and Lemma \ref{filPSS}, $\rmor'$ preserves these filtrations. Denote by $gr^{\mu}\rmor'$ the induced linear map between the associated graded pieces
\[ gr^{\mu}\rmor': gr^{\mu}N_1\ra gr^{\mu}N_2 \]
where $gr^{\mu}N_i:= F^{\mu}N_i/\bigcup_{\lambda<\mu} F^{\lambda} N_i$. By looking at the leading terms, we find that $gr^{\mu}\rmor'$ is injective for any $\mu$, and hence $\rmor'$ is injective, by a filtration argument. Since $\ell':W\ra \RR$ attains its minimum precisely at $e\in W$, it follows that, by the same set of lemmas, $\rmor(S_1)=S_2$. Since $\rmor$ is injective (as $\rmor'$ is) and $S_2$ is a multiplicative subset consisting of central elements (Lemma \ref{prodform}), $S_1$ is a multiplicative subset consisting of central elements as well. This completes the proof of (1).

By Lemma \ref{prodform}, each $F^{\mu}N_2$ is a sub-$\coeff[S_2]$-module. Since $(\rmor')^{-1}(F^{\mu}N_2)=F^{\mu}N_1$ for any $\mu$ (by inspection), each $F^{\mu}N_1$ is a sub-$\coeff[S_1]$-module so that the homomorphism
\[S_1^{-1}(gr^{\mu}\rmor'): S_1^{-1}(gr^{\mu}N_1)\ra gr^{\mu}N_2\]
is well-defined. By Lemma \ref{last}(1) below, $S_1^{-1}(gr^{\mu}\rmor')$ is an isomorphism, and hence $S_1^{-1}\rmor'$ is an isomorphism, by a filtration argument. This proves (2).

Finally, notice that $\coeff[S_1]\simeq \coeff[S_2]\simeq \coeff[S_e]$ where $S_e$ is defined in Lemma \ref{last}. Moreover, each $gr^{\mu}N_1$ is isomorphic as $\coeff[S_e]$-modules to a finite direct sum of $\coeff[S_e]$-modules each of the form $\coeff[S_w]$ (defined below) for some $w\in W$. Thus, (3) follows from Lemma \ref{last}(2) and (3).
\end{proof}

We come back to the lemma used in the proof of Proposition \ref{final}. Recall the Weyl group $W$, the lattice $Q^{\vee}$, the dominant Weyl chamber $\wc_{X_0}$ and the generic element $a$ which lies in the interior $\mr{\wc}_{X_0}$ of $\wc_{X_0}$. For each $w\in W$, define
\[ S_w:= \{q\in Q^{\vee}|~q+w^{-1}a\in \wc_{X_0}\}.\]
Observe that $S_e$ is a semi-group, $S_w\subseteq S_e$ and $S_e+S_w\subseteq S_w$ for any $w\in W$.
\begin{lemma} \label{last} $~$
\begin{enumerate}
\item $S_w-S_e=Q^{\vee}$ for any $w\in W$.
\item $S_e$ is a finitely generated semi-group.
\item $\coeff[S_w]$ is a finitely generated $\coeff[S_e]$-module for any $w\in W$.
\end{enumerate}
\end{lemma}

\begin{proof}
We begin with an observation. Let $\a_1,\ldots, \a_r\in \Ra{X_0}$ be the positive roots which define the boundary walls of $\wc_{X_0}$. Then
\[S_e=\{q\in Q^{\vee}|~\a_i(q)\geqslant 0 ~\forall i\}=\wc_{X_0}\cap Q^{\vee},\]
and for each $w\in W\setminus\{e\}$, there exists a non-empty subset $I_w\subseteq \{1,\ldots, r\}$ such that
\[S_w=\{q\in Q^{\vee}|~\a_i(q)> 0 ~\forall i\in I_w\}\cap S_e.\]
In particular, each $S_w$ contains $\mr{\wc}_{X_0}\cap Q^{\vee}$. This shows (1), since any element of $Q^{\vee}$ becomes regular (i.e. an element of $\mr{\wc}_{X_0}$) after adding a ``sufficiently regular'' element of $Q^{\vee}$. To show (2), observe that each $2\a_i(Q^{\vee})$ is a subgroup of $\ZZ$, by the root space decomposition \eqref{rootsp}. It follows that Gordon's lemma \cite{Fu} implies the result. Finally, (3) simply follows from the fact that $\coeff[S_w]$ is a sub-$\coeff[S_e]$-module of $\coeff[S_e]$.
\end{proof}

\appendix
\section{Some technical lemmas} \label{2lemma}
\subsection{The indices}
\begin{lemma}\label{easy} Let $V,W$ be symplectic vector spaces. Let $L\in\Lag(V) $ and $\RF\in\Lag(W)$ be any Lagrangian subspaces of $V$ and $W$ respectively. The space
\[\lagcor_{L\ra\RF}:=\{C\in \Lag(V^-\oplus W)|~L\circ C=\RF\}\]
is weakly contractible. In particular, if $C\in \lagcor_{L\ra\RF}$ and $L^{\perp}\in \Lag(V)$ such that $L\cap L^{\perp}=\mathbf{0}$, then there is a path in $\lagcor_{L\ra\RF}$, which is unique up to homotopy, joining $C$ and $L^{\perp}\oplus \RF$.
\end{lemma}
\begin{proof}
Observe that every $C\in\lagcor_{L\ra\RF}$ is completely determined by two choices of subspaces $\tilde{\RF}$ and $L'$ where
\begin{enumerate}
\item $\tilde{\RF}$ is a lift of $\RF$ with respect to the symplectic quotient $L\oplus W\ra W$; and
\item $L'$ is a Lagrangian complement of the image of $L$ in the symplectic quotient $\tilde{\RF}^{\perp\w}/\tilde{\RF}$.
\end{enumerate}
($C$ is then the pre-image of $L'$.) Since the space of these choices is a fiber bundle over a contractible base with contractible fibers, it follows that $\lagcor_{L\ra\RF}$ is weakly contractible.
\end{proof}

\begin{lemma} \label{tangentspaceofLag} Let $X\in \t_0$. The tangent space $T_{X_0}\RFF{\exp(X)\cdot p_0}$ of the Lagrangian $\RFF{\exp(X)\cdot p_0}$ at $X_0$ is equal to
\[\bigoplus_{\a\in\Ra{X_0}} \left\{\left. \cos(2\pi \a(X))\eta - \sin(2\pi \a(X)) J_{\CO,X_0}(\eta) \right|~\eta\in\pa{X_0} \right\}  \]
where $J_{\CO,X_0}\in End(T_{X_0}\CO)$ is defined by $J_{\CO,X_0}|_{\gg_{X_0,\a}}:=\frac{-1}{2\pi\a(X_0)}\ad(X_0)|_{\gg_{X_0,\a}}$\footnote{Notice that $Y\mapsto J_{\CO,Y}$ defines a $G$-invariant, $\w_{\CO}$-compatible, integrable almost complex structure on $(\CO,\w_{\CO})$.} and $\gg_{X_0,\a}$ is defined in Section \ref{Lie-basic} (after Remark \ref{Lie-remark}).
\end{lemma}
\begin{proof}
Since $\RFF{\exp(X)\cdot p_0}=\exp(X)\cdot\RF$, we apply decomposition \eqref{rootsp} to obtain the result.
\end{proof}

Put $\ul{x}_0:=([e:X_0],X_0,[e:X_0])$. Let $(u,\g)$ be a \capping disk for $\ul{x}_0$ (see Definition \ref{mcd}). Define a Lagrangian subbundle $\RFF{(u,\g)}$ of $u^*T\CO|_{\pt\DH}$ as follows. Along $[-1,1]$, we put $\RFF{(u,\g)}|_{[-1,1]}\equiv T_{X_0}\RF$. Let $\ol{\g}$ be the projection of $\g$ into $G/K$. For any $z\in \pt_a\DH$, we put $(\RFF{(u,\g)})_z:=T_{u(z)}\RFF{\ol{\g}(z)}$. Define
\begin{align*}
\area(u,\g)& := -\int \g^*\l_M + \int u^*\w_{\CO},\\
\mu(u,\g)& := \mu(u^*T\CO,\RFF{(u,\g)}) .
\end{align*}

\begin{lemma} \label{monocapdisk} Under Assumption \ref{monotone}, every \capping disk for $\ul{x}_0$ satisfies
\begin{equation}\nonumber
\area(u,\g)=\monoconst \cdot \mu(u,\g)
\end{equation}
\end{lemma}
\begin{proof}
It is not hard to see that both $\area(u,\g)$ and $\mu(u,\g)$ are invariant under homotopy of \capping disks for $\ul{x}_0$. Since $\sq:\C\ra \CO$ is a fibration, $u$ has a lift $\tilde{u}$ in $\C$ extending $\g$. This lift allows us to homotope $(u,\g)$ to $(v,\tilde{u}|_{[-1,1]})$ where $v\equiv X_0$. (Strictly speaking, one has to reverse the usual orientation on $[-1,1]$ and reparametrize $\tilde{u}|_{[-1,1]}$ suitably.) Project $\tilde{u}|_{[-1,1]}$ into $T_0$ and lift it to the universal cover $\t_0$. Denote by $q\in Q^{\vee}$ the endpoint of this lifting, assuming it starts at $0$. It is clear that $\area(v,\tilde{u}|_{[-1,1]})=-\langle q,X_0\rangle $. To compute $\mu(v,\tilde{u}|_{[-1,1]})$, we consider the natural trivialization of the symplectic vector bundle $v^*T\CO\simeq T_{X_0}\CO$. By Lemma \ref{tangentspaceofLag}, $\mu(v,\tilde{u}|_{[-1,1]})$ is equal to the degree of the loop
\[t\mapsto \bigoplus_{\a\in\Ra{X_0}} \left\{\left. \cos(2\pi t\a(q))\eta - \sin(2\pi t\a(q)) J_{\CO,X_0}(\eta) \right|~\eta\in\pa{X_0} \right\},\quad t\in [0,1]  \]
in $\Lag(T_{X_0}\CO,(\w_{\CO})_{X_0})$ which is equal to $-2\sum_{\a\in\Ra{X_0}} \ma \a(q)$. The result follows.
\end{proof}

\begin{myproof}{Proposition}{\ref{admiss}} Since every cotangent fiber is contractible, we may assume the labelling does not contain $L_p$. Apply Lemma \ref{easy} so that in the computation of $\mu(\ul{u})$, we may assume every seam condition to be of the form $L_p^{\perp}\oplus\RFF{p}$ where $L_p^{\perp}$ denotes any fiber of the horizontal subbundle of $TM$ with respect to the Levi-Civita connection. As the horizontal subbundle is globally well-defined, the patches labelled by $M$ do not contribute to $\mu(\ul{u})$. On the other hand, the symplectic areas of these patches are equal to the sum of certain integrals of the standard Liouville form $\l_M$. Hence it suffices to verify the monotonicity in the following settings. Let $u:\S\ra \CO$ be a map defined on a compact Riemann surface $\S$ with boundary $\pt\S=\pt_1\S\cup\cdots\cup\pt_k\S$ where each $\pt_i\S$ is a simple closed curve. Suppose $\g_i$ is a lift of $u|_{\pt_i\S}$ in $\C$. These curves determine Lagrangian subbundles of $u^*T\CO|_{\pt_i\S}$. Define $\mu(u,\{\g_i\})$ to be the Maslov index of the corresponding Riemann-Hilbert problem. Define also
\[\area(u,\{\g_i\}):=\int_{\S}u^*\w_{\CO}-\sum_{i=1}^k\int\g_i^*\l_M.   \]
We have to show
\begin{equation}\label{monoprop}
\area(u,\{\g_i\})=\monoconst \cdot \mu(u,\{\g_i\}).
\end{equation}
Since $\CO$ is simply connected, it suffices to verify \eqref{monoprop} when $\S$ is a disk or a sphere. The disk case is just Lemma \ref{monocapdisk}; and the sphere case is in fact a special case of the disk case, by assuming the boundary of the disk to be mapped to the same point.
\end{myproof}

Finally, we prove a formula for quilts which expresses their Fredholm indices in terms of their input and output data. Recall the notations introduced at the beginning of Section \ref{pfA}.

\begin{lemma} \label{Ind} The Fredholm index $\ind(\ul{u})$ of a quilt $\ul{u}$ is given by
\begin{equation}\label{Index}
\ind(\ul{u}) =\sum_{\a\in R_{X_{in}}^+ }\ma \lfloor 2\a(q_{in} +a) \rfloor -   \sum_{\a\in R_{X_{out}}^+ }\ma \lfloor 2 \a(q_{out} +a) \rfloor
\end{equation}
where $\lfloor x \rfloor$ denotes the greatest integer not exceeding $x$.
\end{lemma}

\begin{proof}
Since the \Ham flow $\varphi_H^t$ of $H$ preserves $\C$, we may assume, after taking a gauge transformation $\varphi_H^g$ for some $g:\S_M\ra\RR$, $u_M$ converges to a point at each of the negative strip-like ends (i.e. the time-$\frac{1}{2}$ \Ham chords become time-0 ones). Denote by $\L_M$ and $\RFFF{\CO}$ the vertical tangent bundles of the cotangent bundle $M=T^*(G/K)\ra G/K$ and the fiber bundle $\C\simeq G\times_K \RF \ra G/K$ respectively. Let $\L_M^{\perp}$ be the horizontal subbundle of $TM$ with respect to the Levi-Civita connection. Over $C$, there are two Lagrangian subbundles of the symplectic vector bundle $T(M^-\times \CO)|_C$: the tangent bundle $TC$ and the direct sum $\L_M^{\perp}|_C\oplus \RFFF{\CO}$. (Here we identify $C$ with $\C \subset M$.) Observe that $\L_M\cap \L_M^{\perp}=\mb{0}$ and $\L_M|_C\circ TC=\RFFF{\CO}$. By Lemma \ref{easy}, we can homotope the seam condition $TC$ to $\L_M^{\perp}|_C\oplus \RFFF{\CO}$, such that the non-degeneracy at the negative (triple-)strip-like end is preserved. It follows that the index $\ind(\ul{u})$ is equal to the sum of the indices of the corresponding Fredholm operators over $\S_M$ and $\S_{\CO}$, each with Lagrangian boundary condition. Let us denote them by $\ind_M$ and $\ind_{\CO}$ respectively.

For $\ind_{\CO}$, we proceed as in the proof of Lemma \ref{monocapdisk} and come up with the Lagrangian path
\[t\mapsto \bigoplus_{\a\in\Ra{X_{out}}} \left\{\left. \cos(2\pi t\a(q_{out}+a))\eta - \sin(2\pi t\a(q_{out}+a)) J_{\CO,X_{out}}(\eta) \right|~\eta\in\pa{X_{out}} \right\},\quad t\in [0,1]. \]
By a straightforward computation, we get
\begin{equation}\nonumber
\ind_{\CO} = -\sum_{\a\in\Ra{X_{out}}} \ma \lfloor 2\a(q_{out}+a)\rfloor.
\end{equation}

It remains to compute $\ind_M$. Using the Lagrangian distributions $\L_M$ and $\L_M^{\perp}$ in $M$, we find that $\ind_M$ is equal to minus the cohomological degree of the \Ham chord $x_{in}$. The latter degree can be determined by noticing that along $x_{in}$, the symplectic vector bundle $TM$ splits into $T(T^*T_0)$ and its symplectic orthogonal complement $T(T^*T_0)^{\perp\w}$, and the Lagrangian subbundle $\L_M$ splits into the vertical tangent bundle of $T^*T_0\ra T_0$ and a Lagrangian subbundle of $T(T^*T_0)^{\perp\w}$. It is well-known that the index for the first Lagrangian subbundle is zero. For the second, notice that it plays the same role as $\RFFF{\CO}$, provided that we replace $X_{out}$ by $q_{in}+a$. Since $\phi_0(q_{in}+a)=X_{in}$, we have
\begin{equation}\nonumber
\ind_M = \sum_{\a\in\Ra{X_{in}}} \ma \lfloor 2\a(q_{in}+a)\rfloor.
\end{equation}
\end{proof}

\begin{remark} \label{Indstrip} Similar arguments show that the Fredholm indices of strips and triple-strips have the same form as \eqref{Index}. Details are left to the readers.
\end{remark}
\subsection{A unique continuation lemma}
Since we are restricting ourselves to a smaller class of almost complex structures for defining the Floer-theoretic invariants, the following lemma is necessary in order to classify all holomorphic curves to which the standard transversality arguments cannot be applied, and which we will address by other means.

Let $M^{2m}$ be an even-dimensional manifold (which will be $T^*(G/K)$), and $N\subset M$ an even-dimensional closed submanifold. Let $\mathcal{C}=\{(U_i,\varphi_i)\},~ \varphi_i:U_i\xrightarrow{\sim}  \RR^{2m}$ be a collection of local charts of $M$ such that $\{U_i\}$ covers $N$ and $\varphi_i(U_i\cap N)=V_i$, a subspace of $\RR^{2m}$. Denote by $ \I(M)$ the space of almost complex structures on $M$, and define $\I_{N,\mathcal{C}}(M)$ be the space of $I\in\I(M)$ such that for any $i$, $I|_{U_i}\in\I(\RR^{2m})$ preserves the tangent spaces of every translate of $V_i$. In particular, $N$ is an almost complex submanifold of $M$ with respect to every $I\in\I_{N,\mathcal{C}}(M)$.

We are given the following:
\begin{itemize}
\item a smooth vector field $X$ on $M$ which is tangent to $N$;
\item a connected Riemann surface $\S$, possibly with boundary;
\item a 1-form $\a\in\O^1(\S;\RR)$; and
\item a smooth map $I_{\S}:\S\ra\I_{N,\mathcal{C}}(M)$.
\end{itemize}

\begin{lemma}\label{UCLemma}  Suppose $u:\S\ra M$ is a smooth map satisfying
\begin{equation}\label{UCCR}
(du - X\otimes \a)^{0,1}_{I_{\S}}=0.
\end{equation}
\begin{enumerate}
\item If $u$ maps a non-empty open subset of $\S$ into $N$, then $\im(u)\subset N$.
\item If $u(\pt\S)\subset N$, then $\im(u)\subset N$.
\end{enumerate}
\end{lemma}

\begin{proof}
\begin{enumerate}
\item[]
\item By the standard continuity argument, it suffices to prove the local case, i.e. $\S=D$, the unit disk in $\CC$ (the boundary $\pt \S$ is irrelevant for this part). We may assume $u$ enters $U_i\simeq \RR^{2m}$ for some $i$, $u(0)=0$ and $0$ is a limit point of $u^{-1}(V_i)$. The equation \eqref{UCCR} is transformed into
\begin{equation}\nonumber
\pt_s u(z) + I_{\S}(z,u(z)) \pt_t u(z) + A(z,u(z))=0
\end{equation}
where
\[
\begin{array}{cccl}
A : &D\times \RR^{2m} &\ra& \RR^{2m} \\ \\ [-1em]
 & (z,y) &\mapsto & -\iota_{\pt_s}\a(z)\cdot X(y)-\iota_{\pt_t}\a(z)\cdot I_{\S}(z,y) X(y).
\end{array}
 \]

Let $V_i^{\perp}$ be the orthogonal complement of $V_i$ in $\RR^{2m}$. Let $\pi:\RR^{2m}\ra V_i$ and $\pi^{\perp}:\RR^{2m}\ra V_i^{\perp}$ be the projections. We can write
\[ A(z,y) = A(z,\pi(y)) + B(z,y)\pi^{\perp}(y)\]
where
\[B(z,y):=\int_0^1\pt_yA(z,\pi(y)+t\pi^{\perp}(y)) ~dt \in Hom(V_i^{\perp},\RR^{2m}).\]
By the definition of $\I_{N,\mathcal{C}}(M)$, we see that $\tilde{I}_{\S}(z):= \pi^{\perp}\circ I_{\S}(z,u(z))\circ\iota ,~z\in D$ satisfies $\tilde{I}_{\S}(z)^2=-1$ where $\iota:V_i^{\perp}\ra\RR^{2m}$ is the inclusion. Hence $\tilde{I}_{\S}(z)\in \I(V_i^{\perp})$.

Put $v:=\pi^{\perp}\circ u$. Using the tangency condition on $X$ which implies $\pi^{\perp}A(z,\pi(y))=0$, we have
\begin{enumerate}
\item $v(0)=0$ and $0$ is a limit point of $v^{-1}(0)$; and
\item $v$ solves
\begin{equation}\label{UCCR2}
\pt_sv(z)+\tilde{I}_{\S}(z)\pt_tv(z)+\pi^{\perp}\circ B(z,u(z))v(z) =0.
\end{equation}
\end{enumerate}
We are in a position to apply the Carleman similarity principle \cite{MS} to show that $v\equiv 0$, and hence $u(D)\subset V_i$, as desired.

\item As in the beginning of the proof of (1), we have a smooth map $v:D\cap \{\im(z)\geqslant 0\}\ra V_i^{\perp}$ satisfying $v|_{[-1,1]}\equiv 0$ and equation  \eqref{UCCR2}. Choose smooth extensions of $z\mapsto \tilde{I}_{\S}(z)$ and $z\mapsto \pi^{\perp}\circ B(z,u(z))$ over $D$ and extend $v$ by zero to a function $\tilde{v}$ on $D$. Then $\tilde{v}\in W^{1,p}(D,V_i^{\perp})$ for any $p>2$ and solves \eqref{UCCR2}. The Carleman similarity principle can also be applied so that $u$ maps the upper half disk into $V_i$. Apply (1) to complete the proof.
\end{enumerate}
\end{proof}

For our application, we put $M:=T^*(G/K)$ and $N:=T^*T_0$. We want to find a collection $\mathcal{C}=\{(U_i,\varphi_i)\}$ of local charts as above such that $\J_A(M)\subset \I_{N,\mathcal{C}}(M)$. Pick a point in $T^*T_0$. If it lies in $\UU$, then we choose $(U_i,\varphi_i)$ to be any chart satisfying $U_i\subset \UU$ and that fibers of $\pj$ are identified, via $\varphi_i$, with translates of $V_i$. If that point lies outside $\UU$, it suffices to find a local foliation $\mathcal{F}$ which contains $T^*T_0$ as a leaf and whose leaves are all preserved by $J_0$. (Here we are using the condition that $A\subset \UU$ and is closed in $M$.) This relies on the geometry of $T_0$ and $G/K$.

\begin{lemma}\label{UC} Let $\TTT$ be a totally geodesic submanifold of a Riemannian manifold $(S,g)$ and $V$ a normal vector field defined on $\TTT$. Suppose $V$ is flat as a section of the normal bundle $\mathcal{N}_{S/\TTT}$. Then the submanifold
\[\TTT':=\{(x,\zt+V(x)) |~x\in \TTT, \zt\in T_x\TTT\}\subset T^*S\]
is preserved by $J_g$, where $J_g$ is the canonical almost complex structure associated to the Levi-Civita connection with respect to $g$.
\end{lemma}

\begin{proof}
It suffices to show that for any point $z=(x,\zt+V(x))\in \TTT'$, $T_z\TTT'$ is the direct sum of $T_x\TTT\subset \Ver$ and $T_x\TTT\subset \Hor$, where $T(T^*S)=\Hor\oplus\Ver$ is the natural splitting associated to the Levi-Civita connection. A general path in $\TTT'$ is of the form
\[t\mapsto (x(t),\zt(t)+V(x(t)))\]
where $\zt(t)\in T_{x(t)}\TTT$ is a vector field along $x(t)$. The horizontal component of the derivative of this path is just $\dot{x}(t)$ while the vertical component is $\nabla_{\dot{x}}(\zt+V)= \nabla_{\dot{x}}\zt + \nabla_{\dot{x}}V$. Since $\TTT$ is totally geodesic and $V$ is a flat normal vector field, we have $\nabla_{\dot{x}}\zt, \nabla_{\dot{x}}V\in T\TTT$. The result follows.
\end{proof}

The desired local foliation $\mathcal{F}$ is constructed as follows: WLOG, assume the given point lies in $T_{p_0}^*T_0$ where $p_0:=eK$. First foliate a neighbourhood of $p_0$ in $G/K$ by $g\cdot T_0$ where $g$ runs over a neighbourhood of $e$ in $\exp([\gg,X_0])$. Since every leaf of this foliation is a totally geodesic submanifold with flat normal bundle (Lemma \ref{flattorus}), we can apply Lemma \ref{UC} to obtain our $\mathcal{F}$.

\section{Transversality} \label{trans}
\subsection{Transversality for strips}
Let $u:\RR\times [0,1]\ra M$ be a strip satisfying Conditions \ref{S}. It suffices to show that the only strips to which the standard transversality arguments \cite{Duke} cannot be applied are those which are constant along their length, i.e. independent of the $\RR$-direction. Assume it is the case for $u$. Choose sufficiently large $s_0$ such that $u$ maps $[s_0,+\infty)\times [0,1]$ into a neighbourhood of the Hamiltonian chord $t\mapsto \lim_{s\to +\infty} u(s,t)$ which is contained in $\mr{A}$. It follows that there is an open dense subset $V$ of $[s_0,+\infty)\times [0,1]$ at every point of which $du-X_H\otimes dt$ is tangent to a fiber of $\pj$. These fibers are locally constant on $V$, and hence constant because $\ol{V}=[s_0,+\infty)\times [0,1]$ is connected. By the asymptotic condition, we see that $u$ maps $[s_0,+\infty)\times [0,1]$ into $T^*T_0$. By Lemma \ref{UCLemma}, the image of $u$ lies in $T^*T_0$. To conclude the proof, observe that different \Ham chords $x_q\in\X(L_0,L_1)$ represent different homotopy classes in the space of paths in $T^*T_0$, from $L_0\cap T^*T_0$ to $L_1\cap T^*T_0$. It follows that the limiting \Ham chords of $u$ are equal, and hence $u$ is indeed constant along its length by an action functional argument. The transversality then follows from the non-degeneracy of the \Ham chords of $H$.

\subsection{Transversality for triple-strips} Let $\ul{u}=(u_-,u_{\CO},u_+)$ be a triple-strip satisfying Conditions \ref{TS}. As in the strip case, we assume that the standard transversality argument cannot be applied to $\ul{u}$. Choose sufficiently large $s_0$ such that $u_-$ and $u_+$ map $[s_0,+\infty)\times [0,1/2]$ into a neighbourhood of their limiting \Ham chords which are contained in $\mr{A}$. There are three cases:
\begin{enumerate}
\item There is an open dense subset $V$ of $[s_0,+\infty)\times [0,1/2]$ at every point of which $du_--X_H\otimes dt$ is tangent to a fiber of $\pj$.
\item $u_{\CO}$ is constant.
\item Similar to (1) with $u_-$ replaced by $u_+$.
\end{enumerate}

It suffices to deal with the first two cases. Let us look at Case (1) first. As in the strip case, $u_-$ enters $T^*T_0$. Since $T^*T_0\cap \C$ is the disjoint union of $|W|$ copies of tori and $u_-(\RR\times\{1/2\})\subset\C$, the image of $u_-|_{\RR\times\{1/2\}}$ must lie in one of these tori. By the seam condition, $u_{\CO}|_{\RR\times\{0\}}$ is constant. So by the standard unique continuation theorem (or Lemma \ref{UCLemma} applied to $\CO$ with $N=pt$), $u_{\CO}$ is constant. Thus we are in Case (2).

In this case, $u_{\CO}\equiv wX_0$ for some $w\in W$, and hence $u_-(\RR\times\{1/2\})$ lies in $\sq^{-1}(wX_0)\subset T^*T_0$, by the seam condition. By Lemma \ref{UCLemma}, the image of $u_-$ lies in $T^*T_0$. Observe that the limiting \Ham chords of $u_-$ are equal. We apply an action functional argument to conclude that $u_-$ is constant along its length. The key point is that the integral $\int_{\RR\times\{1/2\}} u_-^*\l_M$ is equal to the pairing of $wX_0$ with the homotopy class represented by the loop defined by projecting $u_-|_{\RR\times\{1/2\}}$ into the zero section $T_0$. Thanks to the boundary condition $u_-(\RR\times\{0\})\subset L_0\cap T^*T_0$, this loop is contractible, and hence the integral is zero.

Similarly, $u_+$ is also constant along its length. Altogether, $\ul{u}$ is constant along its length, and the transversality simply follows from the non-degeneracy of the generalized \Ham chords.

\subsection{Transversality for quilts}
\begin{myproof}{Proposition}{\ref{A}} Let $\ul{u}$ be a quilt. If $\ul{u}$ is not good, then there are two cases:
\begin{enumerate}
\item $u_{\CO}$ is constant.
\item $u_{\CO}$ is non-constant and for any $z\in V\cap(\S_M\setminus\pt_C\S_M)$ with $u_M(z)\in \mr{A}$, $(du_M-X_H\otimes df)(z)$ is tangent to the fiber of $\pj$ containing $u_M(z)$.
\end{enumerate}

In the first case, we have $u_{\CO}\equiv wX_0$ for some $w\in W$. Hence $u_M(\pt_C\S_M)\subset \sq^{-1}(wX_0)\subset T^*T_0$ by the seam condition. By Lemma \ref{UCLemma}, the image of $u_M$ lies in $T^*T_0$. Then $\ul{u}$ is either bad or ugly.

In the second case, draw a smaller disk $V'\subset V$ centered at a point $\in\pt_C\S_M=\pt\S_{\CO}$ such that $u_M(V'\cap \S_M)\subset \mr{A}$. This is possible because $u_M(\pt_C\S_M)\subset\C\subset \mr{A}$. Then $u_M(V'\cap \S_M)$ lies in a fiber of $\pj$. By the seam condition, $u_{\CO}|_{V'\cap\pt\S_{\CO}}$ is constant, and hence $u_{\CO}$ is constant\footnote{See Case (2) in the proof of the transversality for triple-strips.}, a contradiction.
\end{myproof}

\bigskip
\begin{myproof}{Proposition}{\ref{B}} It is clear that the standard transversality argument (proof of \cite[Proposition 3.2.1]{MS}) can be applied to good quilts. This proves the first part of the proposition. A similar argument which uses in addition the proof of \cite[Lemma 3.4.7]{MS} shows that there exists a possibly smaller but still residual subset of the space of regular $\ul{J}$ such that for every good quilt $\ul{u}=(u_M,u_{\CO})$ (with respect to any $\ul{J}$ belonging to this subset), $u_M$ is transverse to $M\setminus\UU$ over the interior of $\S_M$ and to $L_0\setminus\UU$ (resp. $ L_1\setminus\UU$) along the boundary $\pt_{L_0}\S_M$ (resp. $\pt_{L_1}\S_M$). Since these subsets are finite unions of submanifolds of codimension at least two (Lemma \ref{codim2}), the result follows.
\end{myproof}

\bigskip
\begin{myproof}{Proposition}{\ref{B+C/2}} Let $q\in Q^{\vee}$ and $w\in W$. Then the quilt we are looking for must satisfy $u_{\CO}\equiv wX_0$. We find $u_M:\S_M\ra T^*T_0$ satisfying
\begin{itemize}
\item $u_M\to x_q$ at the positive strip-like end;
\item $u_M\to x_{w,\pm}$ at the two negative strip-like ends;
\item $u_M(\pt_{L_0}\S_M)\subset L_0$, $u_M(\pt_{L_1}\S_M)\subset L_1$ and $u_M(\pt_C\S_M)\subset T_0\times\{wX_0\}$; and
\item $(du_M-X_H\otimes df)^{0,1}_{J_0}=0$.
\end{itemize}

\noindent The problem becomes standard if we pass to the universal cover. Before doing this, we define
\[
\begin{array}{cccl}
\tilde{H}: &T^*\t_0 &\ra& \RR\\
&(\eta,\zt) &\mapsto & \frac{1}{2} |\zt|^2
\end{array}.
 \]
Then the \Ham flow $\varphi^t_{\tilde{H}}$ of $\tilde{H}$ is given by
\[ \varphi^t_{\tilde{H}}(\eta,\zt) = (\eta + t\zt, \zt).\]
Define $\afflag_1:=\mb{0}\times \t_0$, $\afflag_2:=\{(q+a-\zt,\zt)|~\zt\in\t_0\}$ and $\afflag_3:=\t_0\times\{wX_0\}$. These are affine linear Lagrangian subspaces of $T^*\t_0$. Identify $\S_M$ with $D\setminus\{1,\pm i\}$ conformally ($D$ being the closed unit disk) such that the positive \sle corresponds to 1 and the others to $\pm i$. Put $J_D(z):= - (  D \varphi^{f(z)}_{\tilde{H}})\circ J_{std} \circ (D\varphi^{-f(z)}_{\tilde{H}}),~z\in D\setminus\{1,\pm i\}$. Observe that $J_D(z)$ is a constant complex structure on $T^*\t_0$ for any $z$. By lifting $u_M$ to the universal cover $T^*\t_0\ra T^*T_0$ and taking the gauge transformation $\varphi^{-f}$, the last problem is equivalent to
\begin{problem}\label{RMT} Find all $v:D\ra T^*\t_0$ satisfying
\[v(e^{i\tt}) \in \left\{
\begin{array}{cl}
\afflag_1,& ~\tt\in[-\frac{\pi}{2},0]\\ \\[-1em]
\afflag_2,& ~\tt\in[0,\frac{\pi}{2}]\\  \\[-1em]
\afflag_3,& ~\tt\in[\frac{\pi}{2},\frac{3\pi}{2}]
\end{array}
\right.\]
and solving (over $\mr{D}$)
\[ (dv)^{0,1}_{J_D} = 0.\]
\end{problem}

\noindent If $J_D$ were domain-independent, then Problem \ref{RMT} would follow immediately from the Riemann mapping theorem. In other words, we need a variant of this theorem. While this should be well-known to experts, we sketch a proof for completeness.

\begin{lemma} For any $q$ and $w$, there exists a unique solution to Problem \ref{RMT}. It is regular and its image lies in the convex hull spanned by the intersection points $\afflag_1\cap \afflag_2$, $\afflag_2\cap \afflag_3$ and $\afflag_1\cap \afflag_3$.
\end{lemma}
\noindent Notice that the projection of the convex hull of these points into $\t_0^{\vee}$ is the line segment joining $q+a$ and $wX_0$. Moreover, this segment does not move if we apply back the gauge transformation $\varphi^{f}$. Thus, the criteria for badness and ugliness stated in Proposition \ref{B+C/2} becomes clear.

\begin{proof}
Observe that the Lagrangians $\afflag_1,\afflag_2$ and $\afflag_3$ are affine linear. Therefore, the difference between two solutions to Problem \ref{RMT} satisfies a linear boundary condition. The uniqueness then follows from the standard energy argument\footnote{Notice we are using the fact that $J_D(z)$ is a constant complex structure for any $z$.}. To prove the existence, first push one of the Lagrangians, say $\afflag_3$, to meet $\afflag_1\cap \afflag_2$ so that Problem \ref{RMT} does have a solution, namely the constant one. This solution is regular by the energy argument as above and the fact that it has index zero (by direct computation). Applying the implicit function theorem, we see that the existence and regularity remain valid if we move $\afflag_3$ slightly away from $\afflag_1\cap \afflag_2$. The general case is then proved by applying a homothety at $\afflag_1\cap \afflag_2$, which works since $J_D(z)$ is a constant complex structure.

Let $v$ be the unique solution to Problem \ref{RMT}. Let $V$ be the unique affine two-dimensional subspace of $T^*\t_0$ passing through $\afflag_1\cap \afflag_2$, $\afflag_2\cap \afflag_3$ and $\afflag_1\cap \afflag_3$. Notice that for any $z$, $J_D(z)$ preserves the tangent spaces of $V$, i.e. the translate of $V$ containing the origin. One may parametrize $V$ by
\[
\begin{array}{ccl}
\RR^2 &\ra& V\\
(x,y) &\mapsto & (x(q+a-wX_0), wX_0+y(q+a-wX_0))
\end{array}.
 \]
Then the pre-images of $\afflag_1$, $\afflag_2$ and $\afflag_3$ are
\[\afflag_1':=\{x=0\},~\afflag_2':=\{x+y=1\}~\text{and}~\afflag_3':=\{y=0\}\]
respectively. By the above arguments, a solution to Problem \ref{RMT} for $(\RR^2;\afflag_1',\afflag_2',\afflag_3')$ exists and hence by the uniqueness, the image of $v$ is contained in $V$.

From now on, we assume $v$ is the solution to the latter problem. It remains to show that the image of $v$ is contained in the triangle
\[ T:= \{0\leqslant x,y,x+y\leqslant 1\}.\]
Suppose not. Take a point $p_0\in \im(v)$ lying outside $T$. Let $h:\RR^2\ra \RR$ be a linear form such that
\begin{enumerate}
\item $h(p_0)>h(p)$ for any $p\in T$; and
\item $\ker(h)$ is not parallel to any of $\afflag_1',\afflag_2',\afflag_3'$.
\end{enumerate}

\noindent Put $\bar{v}:=h\circ v$. Then $\bar{v}$ satisfies a second order elliptic PDE
\[ \Delta \bar{v} + A(z) \pt_s\bar{v} + B(z)\pt_t\bar{v} =0\]
for some smooth functions $A(z)$ and $B(z)$. Let $z_0\in D$ be a point at which $\bar{v}$ attains maximum. By the definition of $h$, $z_0$ cannot be $1,\pm i$. It cannot be an interior point neither, by Hopf's interior maximum principle. So it must be a boundary point $\ne 1,\pm i$. Write $z_0=e^{i\tt_0}$. Then $v$ is smooth at $z_0$ and $\pt_{\tt}v(e^{i\tt_0})$ lies in $\afflag_i\cap \ker(h)$ for some $i=1,2,3$. But $\afflag_i\cap \ker(h)=\mb{0}$ by assumption. It follows that $\pt_{\tt}v(e^{i\tt_0})=0$, and hence $\pt_rv(e^{i\tt_0})=0$, by the Cauchy-Riemann equation. This violates Hopf's boundary maximum principle, a contradiction.
\end{proof}
\end{myproof}

\bigskip
\begin{myproof}{Proposition}{\ref{C}} Write $\phi_0(q+a)=w_qX_0$ where $w_q\in W$. Let $\ul{u}$ be a bad quilt with input $x_q$ (and so output $\ul{x}_{w_q}$). We use the notations introduced at the beginning of Section \ref{pfA}. The key point is that $v=\tilde{u}_{\CO}\# u_M$ is contained in $T^*T_0$ so that the projections of $x_{in}$ and $x_{out}$ into the zero section are homotopic, and hence $q_{in}=q_{out}$. Moreover, $X_{in}=X_{out}$ so that that the two summations in the index formula \eqref{Index} are equal, and hence the index of $\ul{u}$ is zero.

We now prove the regularity. Since the index of $\ul{u}$ is zero, it suffices to show that the kernel of the linearization $D_{\ul{u}}$ of the Cauchy-Riemann operator is zero. Recall $D_{\ul{u}}$ consists of $D_{u_M}$ and $D_{u_{\CO}}$ where $D_{u_M}$ has the following expression
\begin{equation}\label{Du}
D_{u_M}\xi_M = (\nabla \xi_M)^{0,1} - \frac{1}{2} \left( J_M(\nabla_{\xi_M}J_M)(du_M-X_H\otimes df)\right)^{0,1}-(\nabla_{\xi_M} X_H\otimes df)^{0,1}
\end{equation}
and $D_{u_{\CO}}$ is nothing but the standard Cauchy-Riemann operator.

Let $\ul{\xi}=(\xi_M,\xi_{\CO})$ be an element of $\ker(D_{\ul{u}})$. Consider the bundle map $u_M^*T\UU\ra T_{w_qX_0}\CO$ induced by the projection $\pj$. Denote by $\xi'_M$ the image of $\xi_M$ in the trivial bundle $T_{w_qX_0}\CO$. By the seam condition, $\xi'_M$ and $\xi_{\CO}$ define a continuous function on $\S_M\cup \S_{\CO}$ with values in $T_{w_qX_0}\CO$.

We have seen that $\xi_{\CO}$ solves the standard Cauchy-Riemann equation with respect to the domain-dependent almost complex structure $J_{\CO}$. Let us look at what equation $\xi'_M$ solves. Notice that \eqref{Du} does not depend on the metric which induces the connection $\nabla$, since $(du_M-X_H\otimes df)^{0,1}=0$. We may choose the product metric with respect to a local trivialization of $\pj$. Using the assumption that $X_H$ is tangent to the fibers of $\pj$ and $J_M$ preserves their tangent spaces, it is not hard to see that the terms $\left( J_M(\nabla_{\xi_M}J_M)(du_M-X_H\otimes df)\right)^{0,1}$ and $(\nabla_{\xi_M} X_H\otimes df)^{0,1}$ are projected to zero in $T_{w_qX_0}\CO$ and $(\nabla \xi_M)^{0,1}$ to $\frac{1}{2}(\pt_s\xi'_M+J_M'\pt_t\xi'_M)$ where $J'_M$ is the induced almost complex structure on $T_{w_qX_0}\CO$. Therefore, $\xi_M'$ also solves the standard Cauchy-Riemann equation. Since the Lagrangian boundary condition for $\xi'_M$ is locally constant, the standard energy argument shows that $\xi_M'\equiv 0 \equiv \xi_{\CO}$. This implies that $\xi_M$ takes values in $T(T^*T_0)$. The proof is complete by using the regularity of $u_M$ (as a map into $T^*T_0$) which is proved in Proposition \ref{B+C/2}.
\end{myproof}

\bigskip
\begin{myproof}{Proposition}{\ref{D}} Let $\ul{u}$ be an ugly quilt. We have $q_{in}=q_{out}=:q$. Moreover, since the input is a \Ham chord, we have $\a(q+a)>0$ for any $\a\in R_{X_{in}}^+$. (It may not be true for the output.) By the index formula \eqref{Index},
\[\ind(\ul{u}) =\sum_{\a\in R_{X_{in}}^+ }\ma \lfloor 2\a(q +a) \rfloor -  \sum_{\a\in R_{X_{out}}^+ }\ma \lfloor 2\a(q+a) \rfloor.\]
Observe that $R_{X_{out}}^+\setminus R_{X_{in}}^+=-(R_{X_{in}}^+\setminus R_{X_{out}}^+)$. Hence
\begin{equation}\nonumber
\ind(\ul{u}) = \sum_{\substack{\a\in R_{X_{in}}^+ \\ \a(X_{out})<0}} \ma\left( \lfloor 2\a(q +a) \rfloor - \lfloor -2\a(q+a) \rfloor\right)
\end{equation}
which is clearly non-negative, since $ \lfloor \a(q +a) \rfloor \geqslant 0$ and $\lfloor -\a(q+a) \rfloor\leqslant -1$. Since $\ul{u}$ is not bad, we have $X_{in}\ne X_{out}$ (Proposition \ref{B+C/2}). Therefore, the set over which the above summation is taken is non-empty, and hence $\ind(\ul{u})>0$.
\end{myproof}
\section{Signs} \label{sign} This appendix deals with the sign issue which occurs when the coefficient ring $\coeff$ is not equal to $\ZZ_2$. We will follow essentially the general theory, notably \cite{Ab_JSG} and \cite{WWorient}, with one exception: we will not make use of the fact that $\RF$ is $Spin$\footnote{It is $Spin$ because its tangent bundle is stably trivial, see the proof of Lemma \ref{orientable}.}. Instead, we will follow \cite{LagPSS} which takes all homotopy classes of capping disks into account so that the existence of relative $Pin$ structures, whose main use is to orient coherently the determinant lines associated to non-homotopic capping disks, is not required. Of course, the price of gaining such flexibility is that any Floer-related cochain complexes are forced to become infinite dimensional. But this is not a disadvantage for us: enlarging $CF^*(\RF,\RF)$ is exactly our goal in order to obtain Theorem \ref{main}.

Our task consists of redefining $CW_b^*(L,L')$, $Hom_{\AA}^*(\ul{L},\ul{L}')$ and $QC^*(\RF)$ rigorously ($L,L'$ are either $L_0$ or $L_1$, and $\ul{L},\ul{L}'\in\AA$), and describing how signs are defined in any related $A_{\infty}$ operations on them. For the wrapped Floer cochain complex $CW_b^*(L,L')$, we define it in exactly the same way as in \cite{Ab_JSG}. Let us move on to consider $Hom_{\AA}^*(\ul{L},\ul{L}')$. We make the assumption that the flows of any Hamiltonians on $M$ involved preserve the coisotropic submanifold $\C$. This assumption suffices for our purpose. Notice that every such flow induces a \Ham flow on $\CO$, and this allows us to transform our problem to one where for each generalized \Ham chord $\ul{x}=(x_-,x_{\CO},x_+)\in \X(\ul{L},\ul{L}')$, the paths $x_{\pm}$ become points.

Let $\ul{x}\in \X(\ul{L},\ul{L}')$ and $(u,\g)$ be a \capping disk for $\ul{x}$. The lift $\g$ determines a Lagrangian subbundle $\RFF{(u,\g)}$ of $u^*T\CO$ over the boundary arc $\pt_a\DH$ which is defined by $\RFF{(u,\g)}(z):=T_{u(z)}\RFF{\g(z)}$. The \Ham chord $x_{\CO}$ induces a symplectic trivialization of $u^*T\CO$ over the boundary segment which brings the two endpoints of $\RFF{(u,\g)}$ to a pair of transversely intersecting Lagrangian subspaces, by the non-degeneracy condition on $\ul{x}$. The linearization of the Cauchy-Riemann operator associated to $(u^*T\CO,\RFF{(u,\g)})$ is denoted by $\mathcal{F}_{(u,\g)}$. The following lemma is an analogue of Condition (O) in \cite{LagPSS}.

\begin{lemma} Fix a homotopy class $[u,\g]$ of \capping disks for $\ul{x}$. The real line bundle over the space of representatives of $[u,\g]$ with fibers $\det(\mathcal{F}_{(u,\g)})$ is trivial.
\end{lemma}
\begin{proof}
Following \cite{LagPSS}, it amounts to showing that the second Stiefel-Whitney class $w_2(\RFF{\CO})$ of the vertical tangent bundle $\RFF{\CO}$ of the fiber bundle $\C\ra G/K$ vanishes on $\ker(\pi_2(\sq))$. But the last subgroup is zero by $\pi_2(T_0)=0$ and the long exact sequence
\[\cdots \ra \pi_2(T_0)\ra \pi_2(\C)\ra \pi_2(\CO)\ra \cdots.\]
\end{proof}

\noindent It follows that we can speak of the determinant line associated to a given homotopy class of \capping disks for $\ul{x}$, without referring to an explicit representative. We denote this line by $o_{[u,\g]}$. Before defining $Hom_{\AA}^*(\ul{L},\ul{L}')$, we also need to assign determinant lines to the points $x_{\pm}$. This is done under the framework of \cite{Ab_JSG}. Fix gradings on $\Hor$ and $\Ver$, the horizontal and vertical subbundles of $TM$. These gradings give rise to a unique (up to homotopy) Lagrangian path in $T_{x_-}M$ (resp. $T_{x_+}M$) which starts at $(\Ver)_{x_-}$ (resp. $(\Hor)_{x_+}$) and ends at $(\Hor)_{x_-}$ (resp. $(\Ver)_{x_+}$). Then these Lagrangian paths determine determinant lines $o_{x_{\pm}}$.

\begin{definition} (Redefinition of Definition \ref{morphismspace}) The morphism space from $\ul{L}$ to $\ul{L}'$ is defined to be
\[ Hom_{\AA}^*(\ul{L},\ul{L}') := \bigoplus |o_{x_-}\otimes o_{[u,\g]}\otimes o_{x_+}| \]
where the direct sum is taken over all $\ul{x}=(x_-,x_{\CO},x_+)\in\X(\ul{L},\ul{L}')$ and all homotopy classes $[u,\g]$ of \capping disks for $\ul{x}$.
\end{definition}

\noindent The definition of $QC^*(\RF)$ is similar and is left to the readers.

It remains to describe the relevant operations where signs are taken into account. Recall \cite{WWorient} the general way of orienting moduli spaces of pseudoholomorphic quilted maps is to homotope the given seam conditions into split ones, and then apply the standard methods to orient the determinant lines associated to the resulting Lagrangian boundary conditions on the individual patches. In our case, $C\simeq G\times_K\RF$ is the only seam condition. An astonishing feature of $C$ is that there is a foliation of $M$ by cotangent fibers $L_p$ such that $L_p\circ C=\RFF{p}$. By Lemma \ref{easy}, we can homotope $TC$ to $\Hor|_C\oplus \RFF{\CO}$. With this explicit split seam condition, we orient the determinant line associated to the output data as follows. Consider the patches which are labelled by $M$. Notice that the possible Lagrangian boundary conditions for these patches are $\Hor$ or $\Ver$. If we fix in advance relative $Pin$ structures on them as well as the aforementioned Lagrangian paths associated to all possible $x_{\pm}$, then we can orient the output determinant lines following \cite{Ab_JSG}. As for the patches which are labelled by $\CO$, we glue them with all input \capping diks (if there are any). Since the Lagrangian boundary conditions become $\RFF{\CO}$, the Cauchy-Riemann operators on the resulting output \capping disks are exactly the ones which define $o_{[u,\g]}$. Combining the above two cases, we obtain the orientation on the output determinant line.

Concerning operations on $QC^*(\RF)$, we refer to \cite{LagPSS} for a detailed description.


\begin{thebibliography}{99}
\bibitem{AS}
A. Abbondandolo, M. Schwarz: On the Floer homology of cotangent bundles. Comm. Pure Appl. Math. \textbf{59}(2), 254-316 (2006).

\bibitem{Ab_JSG}
M. Abouzaid: On the wrapped Fukaya category and based loops. J. Symplectic Geom. \textbf{10}(1), 27-79 (2012).

\bibitem{Araki} S. Araki: On Bott-Samelson $K$-cycles associated with symmetric spaces. Osaka J. Math. \textbf{13}(2), 87-133 (1962).

\bibitem{BCP} P. Biran, O. Cornea: A Lagrangian quantum homology. In New perspectives and challenges in symplectic field theory, volume 49 of CRM Proc. Lecture Notes, pages 1-44. Amer. Math. Soc., Providence, RI, 2009.

\bibitem{BMich} R. Bott: The space of loops on a Lie group. Michigan Math. J. \textbf{5}, 35-61 (1958).

\bibitem{BS} R. Bott, H. Samelson: Applications of the theory of Morse to symmetric spaces. Amer. J. Math. \textbf{80}(4), 964-1029 (1958).

\bibitem{twist}  F. Burstall, J. Rawnsley: Twistor theory for Riemannian symmetric spaces. Lecture Notes in Math. \text{1424}, Springer (1990).

\bibitem{wonderful} C. de Concini, C. Procesi: Complete symmetric varieties. In Invariant theory (Montecatini, 1982), volume 996 of Lecture Notes in Math., pages 1-44. Springer, Berlin, 1983.

\bibitem{EL}
J. Evans, Y. Lekili: Generating the Fukaya categories of Hamiltonian $G$-manifolds. J. Amer. Math. Soc. \textbf{32}(1), 119-162 (2019).

\bibitem{Floer} A. Floer: Witten's complex and infinite-dimensional Morse theory. J. Differential Geom. \textbf{30}, 207-221 (1989).

\bibitem{Duke}
A. Floer, H. Hofer, D. Salamon: Transversality in elliptic Morse theory for the symplectic action. Duke Math. J. \textbf{80}(1), 251-292 (1995).

\bibitem{FOOO} Fukaya, K., Oh, Y., Ohta, H., Ono, K.: Lagrangian intersection Floer theory - anomaly and obstruction. Part I and II. AMS/IP Stud. Adv. Math. \textbf{46}. American Mathematical Society, Providence, RI (2009).

\bibitem{Fu} W. Fulton: Introduction to toric varieties. Annals of Math. Studies \textbf{131}. Princeton University Press, Princeton (1993).

\bibitem{Ganatra}
S. Ganatra: Symplectic cohomology and duality for the wrapped Fukaya category. Ph.D. thesis, MIT, 2012.

\bibitem{GS} V. Guillemin, S. Sternberg: The moment map revisited. J. Differential Geom. \textbf{69}(1), 137-162 (2005).

\bibitem{BSZ} R. Kocherlakota: Integral homology of real flag manifolds and loop spaces of symmetrical spaces. Adv. Math. \textbf{110}(1), 1-46 (1995).

\bibitem{LS} T. Lam, M. Shimozono: Quantum cohomology of $G/P$ and homology of affine Grassmannian.  Acta Math. \textbf{204}(1), 49-90 (2010).

\bibitem{LL}
Y. Lekili, M. Lipyanskiy: Geometric composition in quilted Floer theory. Adv. Math. \textbf{244}, 268-302 (2013).

\bibitem{LLi1} N.C. Leung, C. Li: Functorial relationships between $QH^*(G/B)$ and $QH^*(G/P)$. J. Differential Geom. \textbf{86}(2), 303-354 (2010).

\bibitem{LLi2} N.C. Leung, C. Li: Gromov-Witten invariants for $G/B$ and Pontryagin product for $\O K$. Trans. Amer. Math. Soc. \textbf{364}(5), 2567-2599 (2012).

\bibitem{LLi3} N.C. Leung, C. Li: Classical aspects of quantum cohomology of generalized flag varieties. Int. Math. Res. Not. IMRN \textbf{2012}(16), 3706-3722 (2012).

\bibitem{MWW}
S. Ma'u, K. Wehrheim, C. Woodward: $A_{\infty}$-functor for Lagrangian correspondences. Selecta Math. (N.S.) \textbf{24}(3), 1913-2002 (2018).

\bibitem{MS}
D. McDuff, D. Salamon: J-holomorphic curves and symplectic topology. Amer. Math. Soc. Colloq. Publ. \textbf{52}. American Mathematical Society, Providence, RI (2004).

\bibitem{Oh} Y.G. Oh: Floer cohomology, spectral sequences, and the Maslov class of Lagrangian embeddings. Int. Math. Res. Not. \textbf{7}, 305-346 (1996).

\bibitem{Peter} D. Peterson: Quantum cohomology of $G/P$. Lecture notes, M.I.T., Spring 1997.

\bibitem{PSS}
S. Piunikhin, D. Salamon, M. Schwarz: Symplectic Floer-Donaldson theory and quantum cohomology. Contact and symplectic geometry. Cambridge University Press, 171-200 (1996).

\bibitem{Warner} G. Warner, Harmonic analysis on semi-simple Lie groups. II. Springer, Berlin (1972).

\bibitem{WW1}
K. Wehrheim, C. Woodward: Quilted Floer cohomology. Geom. Topol. \textbf{14}, 833-902 (2010).

\bibitem{WW1.5}
K. Wehrheim, C. Woodward: Quilted Floer trajectories with constant components: corrigendum to ``Quilted Floer cohomology''. Geom. Topol. \textbf{16}, 127-154 (2012).

\bibitem{WW2}
K. Wehrheim, C. Woodward: Floer cohomology and geometric composition of Lagrangian correspondences. Adv. Math. \textbf{230}(1), 177-228 (2012).

\bibitem{WWorient} K. Wehrheim, C. Woodward: Orientations for pseudoholomorphic quilts. Preprint (2015). \href{https://arxiv.org/abs/1503.07803}{arXiv:1503.07803}.

\bibitem{WW3}
K. Wehrheim, C. Woodward: Pseudoholomorphic quilts. J. Symplectic Geom. \textbf{13}(4), 849-904 (2015).

\bibitem{W} A. Weinstein: Symplectic geometry. Bull. Amer. Math. Soc. (N.S.) \textbf{5}(1), 1-13 (1981).

\bibitem{NASC}
C. Woodward: The classification of transversal multiplicity-free group actions. Ann. Global Anal. Geom. \textbf{14}, 3-42 (1996).

\bibitem{LagPSS} F. Zapolsky: The Lagrangian Floer-quantum-PSS package and canonical orientations in Floer theory. Preprint (2015). \href{https://arxiv.org/abs/1507.02253}{arXiv:1507.02253}.
\end{thebibliography}
\end{document}